\documentclass[a4paper,12pt]{amsart}

\allowdisplaybreaks

\usepackage{latexsym,amsthm,amsmath,amssymb}
\usepackage{mathtools}
\usepackage{color,ulem}

\usepackage{arydshln}
\usepackage{epic,eepic,eepicemu}
\usepackage{multicol}
\usepackage{colortbl}

\usepackage{tikz}
\usepackage{pgfplots}

\usepackage{blkarray}
\usepackage{booktabs
}
\renewcommand{\theenumi}{\roman{enumi}}

\newtheorem{thm}{Theorem}[section]
\newtheorem{prop}[thm]{Proposition}
\newtheorem{lemma}[thm]{Lemma}
\newtheorem{cor}[thm]{Corollary}

\theoremstyle{definition}

\newtheorem{example}[thm]{Example}

\newtheorem*{assumption}{Assumptions}

\newcommand{\cP}{\mathcal{P}}
\newcommand{\cD}{\mathcal{D}}
\newcommand{\cV}{\mathcal{V}}
\newcommand{\cU}{\mathcal{U}}
\newcommand{\cO}{\mathcal{O}}
\newcommand{\cW}{\mathcal{W}}
\newcommand{\cS}{\mathcal{S}}
\newcommand{\cT}{\mathcal{T}}
\newcommand{\cF}{\mathcal{F}}
\newcommand{\cN}{\mathcal{N}}
\newcommand{\cL}{\mathcal{L}}

\renewcommand{\S}{\mathbb{S}}
\renewcommand{\P}{\mathbb{P}}
\newcommand{\B}{\mathbb{B}}
\newcommand{\R}{\mathbb{R}}
\newcommand{\N}{\mathbb{N}}
\newcommand{\C}{\mathbb{C}}
\newcommand{\F}{\mathbb{F}}
\newcommand{\Z}{\mathbb{Z}}

\renewcommand{\phi}{\varphi}
\renewcommand{\epsilon}{\varepsilon}

\newcommand{\vecbf}[1]{\mbox{\boldmath $#1$}}
\newcommand{\vt}{\vecbf{t}}
\newcommand{\vo}{\vecbf{0}}

\newcommand{\mylarge}[1]{\makebox{\large $#1$}}

\newcommand{\bx}{\bar{x}}
\newcommand{\by}{\bar{y}}
\newcommand{\bd}{\bar{d}}
\newcommand{\bv}{\bar{v}}
\newcommand{\blambda}{\bar{\lambda}}
\newcommand{\bX}{\overline{X}}
\newcommand{\bY}{\overline{Y}}
\newcommand{\bZ}{\overline{Z}}

\newcommand{\ty}{\tilde{y}}
\newcommand{\tX}{\widetilde{X}}
\newcommand{\tx}{\tilde{x}}
\newcommand{\tv}{\tilde{v}}
\newcommand{\tg}{\tilde{g}}
\newcommand{\tp}{\tilde{p}}
\newcommand{\tB}{\widetilde{B}}
\newcommand{\tU}{\widetilde{U}}
\newcommand{\tE}{\widetilde{E}}
\newcommand{\tZ}{\widetilde{Z}}
\newcommand{\tM}{\widetilde{M}}
\newcommand{\tf}{\tilde{f}}
\newcommand{\tpsi}{\tilde{\psi}}
\newcommand{\tlambda}{\tilde{\lambda}}
\renewcommand{\th}{\tilde{h}}
\renewcommand{\tt}{\tilde{t}}
\newcommand{\tell}{\tilde{\ell}}
\newcommand{\tgamma}{\tilde{\gamma}}
\newcommand{\tmu}{\tilde{\mu}}

\newcommand{\hx}{\hat{x}}
\newcommand{\hy}{\hat{y}}
\newcommand{\hz}{\hat{z}}
\newcommand{\hh}{\hat{h}}
\newcommand{\hw}{\hat{w}}

\newcommand{\tV}{\widetilde{V}}

\newcommand{\acos}{\mathop\mathrm{Arccos}}
\newcommand{\rank}{\mathop\mathrm{rank}}
\newcommand{\Span}{\mathop\mathrm{Span}}
\newcommand{\tr}{\mathop\mathrm{tr}}
\newcommand{\Gr}{\mathop\mathrm{Gr}}
\renewcommand{\Im}{\mathop\mathrm{Im}}
\newcommand{\rint}{\mathop\mathrm{rint}}
\newcommand{\intr}{\mathop\mathrm{int}}
\newcommand{\cone}{\mathop\mathrm{cone}}
\newcommand{\lt}{\mathop\mathrm{LT}}
\newcommand{\argmin}{\mathop\mathrm{argmin}}
\newcommand{\feas}{\mathop\mathrm{feas}}

\newcommand{\adj}{\mathrm{adj}}
\newcommand{\adjv}{\mathrm{adjv}}
\newcommand{\diag}{\mathrm{diag}}
\newcommand{\supp}{\mathrm{supp}}

\newcommand\coolover[2]{\mathrlap{\smash{\overbrace{\phantom{%
    \begin{matrix} #2 \end{matrix}}}^{\mbox{$#1$}}}}#2}

\newcommand\coolunder[2]{\mathrlap{\smash{\underbrace{\phantom{%
    \begin{matrix} #2 \end{matrix}}}_{\mbox{$#1$}}}}#2}

\newcommand\coolleftbrace[2]{%
#1\left\{\vphantom{\begin{matrix} #2 \end{matrix}}\right.}

\newcommand\coolrightbrace[2]{%
\left.\vphantom{\begin{matrix} #1 \end{matrix}}\right\}#2}

\title[The Characteristic Polynomial and Alternating Projections]{Expansions of the Characteristic Polynomial of a Perturbed PSD Matrix and Convergence Analysis of Alternating Projections for the PSD Cone and a Line}

\author[Y. Sekiguchi]{Yoshiyuki Sekiguchi}
\address[Y. Sekiguchi]{Graduate School of Marine Science and Technology, Etchujima 2-1-8, Koto-ku, Tokyo 135-8533, Japan}
\email[Corresponding author]{yoshi-s@kaiyodai.ac.jp}

\author[H. Ochiai]{Hiroyuki Ochiai}
\address[H. Ochiai, H. Waki]{Institute of Mathematics for Industry, Kyushu University, 744 Motooka, Nishi-ku, Fukuoka 819-0395, Japan}
\email{ochiai@imi.kyushu-u.ac.jp, waki@imi.kyushu-u.ac.jp}

\author[H. Waki]{Hayato Waki}



\begin{document}

\BAtablenotesfalse

\maketitle


\begin{abstract}
We observe that the characteristic polynomial of a linearly perturbed semidefinite matrix can be used to determine the convergence rate of alternating projections
 for the positive semidefinite cone and a line.
As a consequence, we show that such alternating projections converge at $O(k^{-\frac{1}{2}})$, independently of the singularity degree.
A sufficient condition for the linear convergence is also obtained.
Our method directly analyzes the defining equation for an alternating projection sequence without using error bounds.
\end{abstract}

\section{Introduction}

Let $\S^n$ and $\S^n_+$ be the sets of $n\times n$ symmetric matrices and positive semidefinite matrices respectively, and $[n] = \{1,\ldots,n\}$.
For an affine subspace $E$ of $\S^n$, an important task is to find a point in the intersection $E\cap \S^n_+$.
To find such a point, the alternating projection method constructs a sequence $\{U_k\}$ via $U_{k+1} = P_E \circ P_{\S^n_+}(U_k)$ with $U_0\in E$,
where $P_E$ and $P_{\S^n_+}$ are projections onto $E$ and $\S^n_+$ respectively.
If $E\cap \S^n_+$ is nonempty, then
$\{U_k\}$ converges to a point in $E\cap \S^n_+$.
It is well-known that if $E$ intersects with $\S^n_+$ transversely, then $\{U_k\}$ converges linearly \cite{BB1993}. On the other hand, if $E$ intersects with $\S^n_+$ nontransversely, then $\{U_k\}$ converges sublinearly and an upper bound of the rate is given by the singularity degree of $E\cap \S^n_+$ \cite{DLW}.

In this paper, 
we consider the case where $E$ is a line and $E\cap \S^n_+$ is a singleton.
For this case, we seek new upper bounds for the convergence rate of $\{U_k\}$, and examine the tightness of the upper bound.
The key tool is 
the analytic formula in Proposition $\ref{prop:formula_eigen}$
for the parameter in $U_k$.
A more general version of this formula was first obtained in \cite{OSW}.
Suppose $E = \{A + tB:\R\}$ and $E\cap \S^n_+ = \{A\}$, where $A\in \S^n_+$ and $B\in \S^n$. Then Proposition $\ref{prop:formula_eigen}$ gives that
\[
  t_{k+1} = t_k - \frac{1}{\|B\|^2}
 \sum_{i \in n(t_k)}\frac{d}{d t}\frac{1}{2}\lambda_i^2(t_k)
\]
for the parameter $t_k$ in $U_k=A + t_k B$, 
where
$\lambda_i(t_k)$ are the eigenvalues of $A + t_k B$
and $n(t_k) = \{i\in [n]:\lambda_i(t_k)<0\}$.
This formula connects the convergence rate of the alternating projections and the leading terms of the eigenvalues.
Although the eigenvalues of a parametric matrix are hard to obtain, Proposition $\ref{prop:formula_eigen}$ is useful in our case, 
since the leading term of an eigenvalue of a one-parameter matrix is efficiently
determined by the \textit{Newton diagram} associated with its characteristic polynomial.

Thus we first derive an expansion formula for the characteristic polynomial (Proposition $\ref{prop:expansion}$).
Moreover, by considering the Newton polytope of the characteristic polynomial, 
we expand it further using sums of squares of minors
 (Theorem $\ref{thm:generic}$) and obtain sufficient conditions for the coefficients to be zero (Theorem $\ref{thm:degenerate}$).
We then show that the leading term of every eigenvalue of a one-parameter perturbation of a positive semidefinite matrix has degree less than or equal to $2$, using the Newton diagram (Theorem $\ref{thm:degree}$).

These results are applied to the convergence analysis of the alternating projections via Proposition $\ref{prop:formula_eigen}$. The following is the main theorem.
\begin{thm}
\label{thm:convergence0}
Let $\{U_k\}$ be the alternating projections for $\S^n_+$ and a line 
$E = \{A + tB:t\in \R\}$, where $A\in \S^n_+$ and $B\in \S$.
If $\S^n_+ \cap E = \{A\}$, then $\|U_{k} - A\| = O(k^{-\frac{1}{2}})$. 
\end{thm}
Theorem $\ref{thm:convergence0}$ ensures that $O(k^{-\frac{1}{2}})$ is an upper bound for the convergence rate
of the alternating projection method, independent of the singularity degree.
We will prove Theorem $\ref{thm:convergence0}$ as a corollary of Theorem $\ref{thm:convergence}$.
Moreover, Theorem $\ref{thm:convergence}$ implies that if a submatrix of the perturbing matrix $B$ satisfies a rank condition, then
the alternating projections converge linearly.
We also show that the upper bound $O(k^{-\frac{1}{2}})$ is tight if 
the singularity degree is $2$ (Proposition $\ref{prop:sd2}$).

The paper is organized as follows. Section $\ref{section:prelim}$ provides the basic notation and some examples. The expansion formula for the characteristic polynomial for a general matrix is given in Section $\ref{section:expansion}$. 
Section $\ref{section:pminor}$ presents a formula for the sum of the principal minors.
The formula for the characteristic polynomial 
is further expanded
using the Newton polytope in Section $\ref{section:newton_polytope}$. Section $\ref{section:newton_diagram}$ estimates the leading degrees of the eigenvalues via Newton diagram associated with the characteristic polynomial.
Section $\ref{section:ap}$ deals with the convergence analysis of alternating projections.

\section{Preliminaries}
\label{section:prelim}
\subsection{Basic notation}
For $n\times n$ symmetric matrices $U$ and $V$, $\langle U, V\rangle = \tr(UV)$ and $\|U\| = \sqrt{\langle U, U \rangle}$.
Define the projection $P_H(U)$ of $U$ onto a subset $H$ of $\S^n$ by $P_H(U) = \argmin_{X\in H}\|X - U\|$.
We also consider a general $n\times m$ matrix $A$.
For $\gamma\in [n]$, let $|\gamma|$ and $\#\gamma$ be the sum of the elements and the number of elements of $\gamma$, respectively.
For $\alpha\subset[m],\beta\subset [n]$,
we denote by $A[\alpha,\beta]$ the submatrix whose entries are those in
the rows of $A$ indexed by $\alpha$ and the columns of $A$ indexed by $\beta$.
In the case $\alpha = \beta$, we simply write $A[\alpha,\alpha]$ as $A[\alpha]$.
For $N = \min\{m,n\}$, we denote by $\sum_{N\times N}|A|^2$ the sum of the squares of $N\times N$ minors of $A$. 
If $A$ is a square matrix, we denote by $\sum_{d\times d}|A|$ the sum of $d\times d$ principal minors of $A$. We define $\sum_{0\times 0}|A| = \sum_{0\times 0}|A|^2 = 1$.

In the following, 
we define a compound matrix, a higher order adjugate matrix and other related matrices.
See \cite[Section $0.8.1,\ 0.8.12$]{HJ} for basic properties and examples below.

\subsection{Compound matrices.}
For $\alpha \subset [m]$, we denote by $\langle \alpha \rangle_k$ the partially ordered set which consists of all the subsets of $\alpha$ with $k$ elements and are ordered lexicographically. 
We define $\langle m \rangle_k = \langle [m]\rangle_k$.
For $A\in \R^{m\times n}$ and $k\leq \min\{m,n\}$, the \textit{$k$-th compound matrix} $C_k(A)$ is the 
$\binom{m}{k}\times \binom{n}{k}$
matrix whose $(\alpha,\beta)$ entry is 
\[
 C_k(A)_{\alpha,\beta} = \det A[\alpha,\beta]\ \text{for all }\alpha 
\in \langle m \rangle_k,\ \beta \in \langle n \rangle_k.
\]
We define $C_0(A)=1$ and $C_k(A) = 0$ for $k < 0$.
For $d\leq k\leq \min\{m,n\}$, we denote by $C_k^d(A)$ the $\binom{m}{k}\times \binom{n}{k}$ matrix 
whose $(\alpha,\beta)$ entry is 
\[
C_k^d(A)_{\alpha,\beta} = \sum_{d\times d}|A[\alpha,\beta]|\ \text{ for }\alpha\in \langle m\rangle_k,\ \beta\in \langle n\rangle_k. 
\]
For $d < 0$, we define $C_k^d(A)$ as the zero matrix of size $\binom{m}{k}\times \binom{n}{k}$. 
In addition, for $m=n$, we define $C_k^0(A)$ as the identity matrix of size $\binom{n}{k}$.

\subsection{$k$-th adjugate matrices.}
For $m=n$ and $0<k<n$, the \textit{$k$-th adjugate matrix} $\adj_k(A)$ is the $\binom{n}{k}\times \binom{n}{k}$ matrix whose 
$(\alpha,\beta)$ entry is 
\[
 \adj_k(A)_{\alpha,\beta} = (-1)^{|\alpha| + |\beta|}
\det A[\beta^c,\alpha^c], \text{for all }\alpha,\beta 
\in \langle n \rangle_k.
\]
where $\alpha^c = [n]\setminus \alpha$.
We see that $\adj_0(A) = \det A$ and $\adj_1(A)$ is the standard adjugate matrix $\adj A$.
We define $\adj_n(A) = 1$ and $\adj_k(A) = 0$ for $k<0$.
For $m < n$, \textit{the adjugate vector} $\adjv(A)$ is the 
$\binom{n}{n-m}$ row vector whose entry is
\[
 \adjv(A)_{\alpha} = (-1)^{|\alpha|}\det A[[m], \alpha^c], \text{ for all }\alpha \in \langle n \rangle_{n-m}.
\]
For $m > n$, $\adjv(A)$ is similarly defined as the corresponding column vector.

\begin{example}
\label{ex:compound}
Let
$A = 
\left(
\begin{smallmatrix}
a_{11} & a_{12} & a_{13}\\ 
a_{21} & a_{22} & a_{23}\\
a_{31} & a_{32} & a_{33}
\end{smallmatrix}
\right)
$.
Then
$
C_1(A) = A,\ 
C_0(A) = 1, \ C_3(A) = \det B$, 
{\scriptsize
\[
C_2(A)=
\begin{blockarray}{cccc} 
 [1,2] & [1,3] & [2,3] \\
\begin{block}{(ccc)c}
  \begin{vmatrix} a_{11} & a_{12}\\ a_{21} & a_{22} \end{vmatrix} & 
 \begin{vmatrix} a_{11} & a_{13}\\ a_{21} & a_{23} \end{vmatrix} &
 \begin{vmatrix} a_{12} & a_{13}\\ a_{22} & a_{23} \end{vmatrix} & [1,2] \\[1.1em]
 \begin{vmatrix} a_{11} & a_{12}\\ a_{31} & a_{32} \end{vmatrix} &
 \begin{vmatrix} a_{11} & a_{13}\\ a_{31} & a_{33} \end{vmatrix} &
 \begin{vmatrix} a_{12} & a_{13}\\ a_{32} & a_{33} \end{vmatrix} & [1,3] \\[1.1em]
 \begin{vmatrix} a_{21} & a_{22}\\ a_{31} & a_{32} \end{vmatrix} &
 \begin{vmatrix} a_{21} & a_{23}\\ a_{31} & a_{33} \end{vmatrix} &
 \begin{vmatrix} a_{22} & a_{23}\\ a_{32} & a_{33} \end{vmatrix} & [2,3] \\
\end{block} 
\end{blockarray},\quad
C_2^1(A) = 
\begin{blockarray}{cccc}
 [1,2] & [1,3] & [2,3] \\
\begin{block}{(ccc)c}
a_{11}+a_{22} & a_{11} + a_{23} & a_{12} + a_{23} & [1,2]\\
a_{11} + a_{32} & a_{11} + a_{33} & a_{12} + a_{33} & [1,3]\\
a_{21} + a_{32} & a_{21} + a_{33} & a_{22} + a_{33} & [2,3]\\
\end{block}
\end{blockarray}
\]
}
$C_2^2(A) = C_2(A)$. Note that $C_2^0(A)$ is the identity matrix of size $3$, $C_3^0(A) = 1$, $C_3^1(A) = \tr(A)$, $C_3^2(A) = \sum_{2\times 2}|A|$.
%
%
\end{example}
\begin{example}
 Let $A = 
\left(
\begin{smallmatrix}
a_{11} & a_{12} & a_{13}\\ 
a_{21} & a_{22} & a_{23}\\
a_{31} & a_{32} & a_{33}
\end{smallmatrix}
\right)$.
Then $\adj_0(A) = \det A,\ \adj_3(A) = 1$,
{\scriptsize
\[
\adj_1(A) = 
 \begin{blockarray}{cccc}
  [1] & [2] & [3] \\
\begin{block}{(ccc)c}
  \begin{vmatrix} a_{22} & a_{23}\\ a_{32} & a_{33} \end{vmatrix} &
 -\begin{vmatrix} a_{12} & a_{13}\\ a_{32} & a_{33} \end{vmatrix} & 
  \begin{vmatrix} a_{12} & a_{13}\\ a_{22} & a_{23} \end{vmatrix} & [1] \\[1em]
 -\begin{vmatrix} a_{21} & a_{23}\\ a_{31} & a_{33} \end{vmatrix} & 
  \begin{vmatrix} a_{11} & a_{13}\\ a_{31} & a_{33} \end{vmatrix} & 
 -\begin{vmatrix} a_{11} & a_{13}\\ a_{21} & a_{23} \end{vmatrix} & [2] \\[1em]
  \begin{vmatrix} a_{21} & a_{22}\\ a_{31} & a_{32} \end{vmatrix} & 
 -\begin{vmatrix} a_{11} & a_{12}\\ a_{31} & a_{32} \end{vmatrix} & 
  \begin{vmatrix} a_{11} & a_{12}\\ a_{21} & a_{22} \end{vmatrix} & [3]\\
\end{block}
 \end{blockarray},\
\adj_2(A) = 
\begin{blockarray}{cccc}
[1,2] & [1,3] & [2,3] \\
 \begin{block}{(ccc)c}
a_{33} & -a_{23} & a_{13} & [1,2] \\
-a_{32} & a_{22} & -a_{12} & [1,3]\\
a_{31} & -a_{21} & a_{11} & [2,3]\\
 \end{block}
\end{blockarray}
\] 
}
\end{example}
\begin{example}
 Let $A = 
\left(
\begin{smallmatrix}
a_{11} & a_{12} & a_{13}\\ 
a_{21} & a_{22} & a_{23}
\end{smallmatrix}
\right)$. \\
Then 
{\small $ \adjv(A) = 
\left(
\begin{vmatrix} a_{12} & a_{13}\\ a_{22} & a_{23}  \end{vmatrix},\ 
 -\begin{vmatrix} a_{11} & a_{13}\\ a_{21} & a_{23}  \end{vmatrix},\  
\begin{vmatrix} a_{11} & a_{12}\\ a_{21} & a_{22} \end{vmatrix}
\right)
$}.
\end{example}

\section{Characteristic polynomial}
\label{section:charpoly}

\subsection{Expansion of the characteristic polynomial of a perturbed diagonal matrix}
\label{section:expansion}

We have the following expansion of the determinant for the general matrices $A,B\in \R^{n\times n}$.
\begin{lemma}
[\cite{HJ}]
\label{lemma:basic}
$
   \det(A + tB) = \sum_{i=0}^n \langle \adj_i(A), C_i(B)\rangle t^i.
$
\end{lemma}
By the definitions of $\adj_i(A)$ and $C_i(B)$, we see that
\[
  \langle \adj_i(A), C_i(B)\rangle
= \sum_{\alpha,\beta\in \langle n\rangle_i} (-1)^{|\alpha| + |\beta|}\det A[\beta^c,\alpha^c] \det B[\alpha,\beta].
\]
Here, we use the convention $\det A[\emptyset,\emptyset] = 1$.
In particular, if $A$ is diagonal, then $\adj_j(A)$ is also diagonal and we can write
\[
    \langle \adj_i(A), C_i(B)\rangle = \sum_{\alpha \in \langle n\rangle_i} \det A[\alpha^c] \det B[\alpha]
 = \sum_{\alpha \in \langle n\rangle_{n-i}} \det A[\alpha] \det B[\alpha^c].
\]
A direct application of Lemma $\ref{lemma:basic}$ to a diagonal matrix $A$ gives the following proposition, which is needed in later sections.
Recall that we defined $C_{i+j}^j(B)$ as the identity matrix of size $\binom{n}{i+j}$ for $j = 0$, 
and $C_{i+j}^j(B) = 0$ for $j< 0$.
\begin{prop}
\label{prop:expansion}
For $m,n\in \N$ with $m<n$, and $a_k>0\ (k\in [n-m])$, let $A = \diag(a_1,\ldots, a_{n-m},0,\ldots,0)$, $B\in \S^n$. Then the characteristic polynomial $p_{A + tB}(x)$ of $A + tB$ is written by
\[
   p_{A + tB}(x) 
 = 
  \sum_{i=0}^n \sum_{j=m-i}^{n-i}
(-1)^{n-i}\langle \adj_{i+j}(A), C_{i+j}^j(B)\rangle\, t^jx^i.
\]
\end{prop}
To understand the proof better, let us first consider an example for $n=4,\ m = 2$.
\begin{example} 
\label{ex:expansion}
For $a_1,a_2>0$, let
{\scriptsize
 \[
 A = \begin{pmatrix} a_1 & 0 & 0 & 0\\ 0 & a_2 & 0 & 0\\ 0 & 0 & 0 & 0\\ 0 & 0 
& 0 & 0 \end{pmatrix},\quad
B = \begin{pmatrix} 
b_{11} & b_{12} & b_{13} & b_{14}\\ 
b_{12} & b_{22} & b_{23} & b_{24}\\ 
b_{13} & b_{23} & b_{33} & b_{34}\\ 
b_{14} & b_{24} & b_{34} & b_{44}
 \end{pmatrix} 
\]}
Then Proposition $\ref{prop:expansion}$ gives that 
{\scriptsize 
\begin{align*} 
& p_{A + tB}(x)  \\
& = \langle \adj_{4}(A), C_{4}^0(B)\rangle
x^4 
- \left(
\langle \adj_{3}(A), C_{3}^0(B)\rangle 
+ \langle \adj_{4}(A), C_{4}^1(B)\rangle t
\right) x^3 \\
& \hspace{1em} + \left(
\langle \adj_{2}(A), C_{2}^0(B)\rangle 
+ \langle \adj_{3}(A), C_{3}^1(B)\rangle t
+ \langle \adj_{4}(A), C_{4}^2(B)\rangle t^2
\right) x^2\\
& \hspace{1em}- \left(
\langle \adj_{1}(A), C_{1}^0(B) \rangle 
+ \langle \adj_{2}(A), C_{2}^1(B) \rangle t
+ \langle \adj_{3}(A), C_{3}^2(B) \rangle t^2
+ \langle \adj_4(A), C_4^3(B)\rangle t^3
\right) x \\
& \hspace{1em}+ \left(
\langle \adj_{0}(A), C_{0}^0(B)\rangle
+ \langle \adj_{1}(A), C_{1}^1(B)\rangle t
+ \langle \adj_{2}(A), C_{2}^2(B)\rangle t^2
+ \langle \adj_{3}(A), C_{3}^3(B)\rangle t^3
+ \langle \adj_{4}(A), C_{4}^4(B)\rangle t^4
\right)
\\
& = x^4-\left(
a_1 + a_2 
+ \left(b_{11} + b_{22} + b_{33} + b_{44}\right)t
\right)x^3 \\
& \hspace{1em} + \left(
a_1 a_2
+ 
\left(
a_1(b_{22} + b_{33} + b_{44}) + a_2(b_{11} + b_{33} + b_{44})
\right)t
+ 
\sum_{2\times 2}
\begin{vmatrix}
 B
\end{vmatrix}
t^2 
\right)x^2
\\
& \hspace{1em}-\left(
a_1 a_2(b_{33} + b_{44})t
+ \left( 
a_2\sum_{2\times 2}
\begin{vmatrix}
b_{11} & b_{13} & b_{14}\\ 
b_{13} & b_{33} & b_{34}\\ 
b_{14} & b_{34} & b_{44}
 \end{vmatrix}
+ 
a_1\sum_{2\times 2}
\begin{vmatrix}
b_{22} & b_{23} & b_{24}\\ 
b_{23} & b_{33} & b_{34}\\ 
b_{24} & b_{34} & b_{44} 
   \end{vmatrix}
\right)t^2
+
\sum_{3\times 3}
\begin{vmatrix} 
B
\end{vmatrix}t^3
\right)x\\
& \hspace{1em}
+
a_1 a_2 \begin{vmatrix}
  b_{33} & b_{34} \\
  b_{34} & b_{44}
 \end{vmatrix}t^2
+ 
 \left(
a_2 \begin{vmatrix}
b_{11} & b_{13} & b_{14}\\ 
b_{13} & b_{33} & b_{34}\\ 
b_{14} & b_{34} & b_{44}
 \end{vmatrix}
+
 a_1 \begin{vmatrix}
b_{22} & b_{23} & b_{24}\\ 
b_{23} & b_{33} & b_{34}\\ 
b_{24} & b_{34} & b_{44} 
   \end{vmatrix}
\right)t^3
+ 
\begin{vmatrix} 
 B
\end{vmatrix}t^4.
\end{align*}}
\end{example}
\begin{proof}[Proof of Proposition $\ref{prop:expansion}$]
Let $I\in \R^{n\times n}$ be the identity matrix. By repeatedly applying Lemma $\ref{lemma:basic}$, we obtain
\begin{align*}
  p_{A + tB}(x) & = \det(x I - A - tB) 
 = \sum_{i=0}^n \langle \adj_{n-i}(x I), C_{n-i}(-A - tB)\rangle\\
& = \sum_{i=0}^n (-1)^{n-i}\tr C_{n-i}(A + tB) x^i
 =  \sum_{i=0}^n (-1)^{n-i} \sum_{\alpha \in \langle n \rangle_{n-i}}\det((A + tB)[\alpha]) x^i\\
& = \sum_{i=0}^n (-1)^{n-i}\sum_{\alpha\in \langle n\rangle_{n-i}}\sum_{j=0}^{n-i}
\langle \adj_j(A[\alpha]),C_j(B[\alpha])\rangle t^j x^i.
\end{align*}
Since $A = \diag(a_1,\ldots,a_{n-m},0,\ldots,0)$, we see that
$\adj_j(A[\alpha]) = O$ for $\#\alpha \geq n - m + 1$ and $j = 0,1,\ldots, \#\alpha - n + m - 1$. Thus if $\#\alpha = n-i$ for some $i = 0,\ldots,m-1$, then $\adj_j(A[\alpha]) = O$ for $j = 0,\ldots,m-i-1$.
We recall that $\adj_j(A) = 0$ and $C_j(A) = 0$ for $j<0$.
Then we have
\begin{align*}
p_{A + tB}(x)   
& = \left(\sum_{i=0}^{m-1} + \sum_{i=m}^n \right)(-1)^{n-i}\sum_{\alpha\in \langle n\rangle_{n-i}}\sum_{j=0}^{n-i}\langle \adj_j(A[\alpha]),C_j(B[\alpha])\rangle t^j x^i\\
& = \sum_{i=0}^{m-1} (-1)^{n-i}\sum_{\alpha \in \langle n \rangle_{n-i}}\sum_{j=m-i}^{n-i}
\langle \adj_j(A[\alpha]),C_j(B[\alpha])\rangle t^j x^i\\
& \hspace{3em}+ \sum_{i=m}^n (-1)^{n-i}\sum_{\alpha \in \langle n \rangle_{n-i}}\sum_{j=0}^{n-i}
\langle \adj_j(A[\alpha]),C_j(B[\alpha])\rangle t^j x^i\\
& = \sum_{i=0}^n (-1)^{n-i}\sum_{\alpha \in \langle n \rangle_{n-i}}\sum_{j=m-i}^{n-i}
\langle \adj_j(A[\alpha]),C_j(B[\alpha])\rangle t^j x^i\\
& = \sum_{i=0}^n (-1)^{n-i}\sum_{j=m-i}^{n-i}\sum_{\alpha \in \langle n\rangle_{n-i}}
\langle \adj_j(A[\alpha]),C_j(B[\alpha])\rangle t^j x^i.
\end{align*}
Let $i = 0,\ldots,n,\ j = m-i,\ldots,n-i$ be fixed. 
To obtain the result, we will show 
\begin{equation}
  \sum_{\alpha \in \langle n\rangle_{n-i}}
\langle \adj_j(A[\alpha]),C_j(B[\alpha])\rangle
= \langle \adj_{i+j}(A), C_{i+j}^j(B)\rangle. \label{eq:adjcomp}
\end{equation}
Since $\adj_j(A[\alpha])$ is diagonal, the definition implies
\[
   \langle \adj_j(A[\alpha]),C_j(B[\alpha])\rangle 
 = \sum_{\beta \in \langle \alpha\rangle_j}
\det(A[\alpha\setminus\beta])\det(B[\beta]).
\]
Since $A = \diag(a_1,\ldots,a_{n-m},0,\ldots,0)$, for $\beta\subset \alpha$,
we have $\det(A[\alpha\setminus\beta])\neq 0$ if and only if
$\alpha\cap \{n - m+1,\ldots,n\} = \beta\cap \{n - m+1,\ldots,n\}$.
Thus by putting
\[
  D = \{(\alpha,\beta)\subset [n]\times [n]: \beta \subset \alpha,\ 
\#\alpha = n-i,\ \#\beta = j,\ \alpha\setminus [n-m] = \beta \setminus [n-m]\},
\]
we can write
\[
  \sum_{\alpha \in \langle n\rangle_{n-i}}
\langle \adj_j(A[\alpha]),C_j(B[\alpha])\rangle
= \sum_{(\alpha,\beta)\in D}\det(A[\alpha\setminus\beta])\det(B[\beta]).
\]
Let 
\[
  D' = \{(\gamma,\beta)\subset [n]\times [n]: \gamma \subset [n-m],\ \beta \subset [n]\setminus\gamma,\ 
\#\gamma = n-i-j,\ \#\beta = j\}.
\]
Then $(\alpha,\beta)\in D$ if and only if $(\alpha\setminus \beta, \beta)\in D'$.
Thus we have
\begin{align*}
& \sum_{(\alpha,\beta)\in D}\det(A[\alpha\setminus\beta])\det(B[\beta])
 = \sum_{(\gamma,\beta)\in D'}\det(A[\gamma])\det(B[\beta])\\
& = \sum_{\overset{\gamma \subset [n-m]}{\#\gamma = n-i-j}}
\sum_{\overset{\beta\subset [n]\setminus \gamma}{\#\beta = j}}\det(A[\gamma])\det(B[\beta])
 = \sum_{\gamma \in \langle n-m\rangle_{n-i-j}}
\det(A[\gamma])\sum_{j\times j}
\begin{vmatrix}
 B[\gamma^c]
\end{vmatrix}.
\end{align*}
Now the diagonality of $\adj_{i+j}(A)$ and the definition of $C_{i+j}^j(B)$ give that
\[
\langle \adj_{i+j}(A),C_{i+j}^j(B)\rangle
= 
\sum_{\alpha \in \langle n\rangle_{i+j}}
\det(A[\alpha^c])\sum_{j\times j}
\begin{vmatrix}
 B[\alpha]
\end{vmatrix}.
\]
Since $\det(A[\alpha^c]) = 0$ if $\alpha\cap \{n-m+1,\ldots,n\} \neq \emptyset$, we obtain
\[
\sum_{\alpha \in \langle n\rangle_{i+j}}
\det(A[\alpha^c])\sum_{j\times j}
\begin{vmatrix}
 B[\alpha]
\end{vmatrix}  = 
\sum_{\gamma \in \langle n-m\rangle_{n-i-j}}
\det(A[\gamma])\sum_{j\times j}
\begin{vmatrix}
 B[\gamma^c]
\end{vmatrix}.
\]
Therefore, we have shown the equation $(\ref{eq:adjcomp})$.
This completes the proof.
\end{proof}

\subsection{The sum of the principal minors}
\label{section:pminor}
We further expand an entry of $C_{i+j}^j(B)$
in the expansion in Proposition $\ref{prop:expansion}$, which is
the sum of principal minors.
For $m,\ell\in \N$,
consider the matrix
{\small %
\begin{equation}
  M = 
\begin{array}[b]{cl}
\phantom{(}
\begin{array}{cc}
 \mathmakebox[1em]{\ell} & \mathmakebox[3em]{m}
\end{array}
\phantom{)}
& \\
\left(
     \begin{array}{@{}c@{\,}|ccc}      
     M_{11} & & M_{12} & \\
      \hline \\
     M_{21} & & M_{22} \\
                &
     \end{array} 
\right) 
& \hspace{-1em}
     \begin{array}{l}
      \ell \\
      \\
      m\\
      \phantom{1}
     \end{array} 
\end{array}.\label{eq:partition}
\end{equation}} %
We assume $M_{22} = CDF$, where $r\in \N\cup \{0\}$ with $r + \ell \leq m$, $C\in \R^{m\times r}$, $D\in \R^{r\times r}$, $F\in \R^{r\times m}$.
Here, we note $\rank (CDF)\leq r$.
We put
{\small %
\begin{equation}
  X 
= 
\begin{array}[b]{cl}
\phantom{(}
\begin{array}{ccc}
 \mathmakebox[1em]{\ell} & \mathmakebox[1em]{\ell} & \mathmakebox[1em]{r}
\end{array} 
\phantom{)}
& \\
\left(
     \begin{array}{@{\,}c|@{\,}c@{\,}|c@{\,}}      
     I & O & O\\
      \hline
    & & \\[0.4em]
     O & M_{21} & C \\[0.4em]
    & & 
     \end{array} 
\right) 
& \hspace{-1em}
     \begin{array}{l}
      \ell \\
     \\[0.4em]
      m \\[0.4em]      
      \phantom{1}
     \end{array}
\end{array},\
Y = 
\begin{array}[b]{cl}
\phantom{(}
\begin{array}{ccc}
 \hspace{.5em} \mathmakebox[1em]{\ell} & \mathmakebox[1em]{\ell} & \mathmakebox[1em]{r}
\end{array} 
\phantom{)}
&  \\
\left(
     \begin{array}{@{\,}c@{\,}|c|c@{\,}}      
     M_{11} & I & O\\
      \hline
     I & O & O\\
      \hline
     O & O & D
     \end{array} 
\right) 
& \hspace{-1em}
     \begin{array}{l}
      \ell \\
      \ell \\  
      r
     \end{array}
\end{array},\
Z = 
\begin{array}[b]{cl}
\phantom{(}
\begin{array}{cc}
  \mathmakebox[1em]{\ell} \hspace{1em} & \mathmakebox[3em]{m} 
\end{array} 
\phantom{)}
& \\
\left(
     \begin{array}{c|ccc}      
     I & & O & \\
      \hline
     O & & M_{12} & \\
      \hline
     O & & F & 
     \end{array} 
\right) 
& \hspace{-1em}
     \begin{array}{l}
      \ell \\
      \ell \\  
      r
     \end{array}
\end{array}.\label{eq:xyz}
\end{equation}} %
In the case $r = 0$, there are no blocks corresponding to $C,D$ or $F$.
%
%
Then we obtain $M = XYZ$.
With this relation, we have the following factorization for the $(m-r-\ell)$-th adjugate matrix of $M$.
%
%
\begin{lemma}
\label{lemma:adj}
In case $r>0$, we have
\[
  \adj_{m - r - \ell}(M) =
\begin{pmatrix}
O & O \\
O & \tM 
\end{pmatrix} 
\in \R^{\binom{m+\ell}{r+2\ell}\times \binom{m+\ell}{r+2\ell}},
\]
where
 \[
  \tM
= (-1)^\ell 
\det D\cdot
\adjv
\left(
\begin{array}{@{}c@{\,}c@{}}
 \\
 M_{21} & C\\
 \phantom{1}
\end{array}
\right)
\adjv
\left(
\begin{array}{ccc}
 &  M_{12} & \\
 &  F &
\end{array}
\right) \in \R^{\binom{m}{r+\ell}\times \binom{m}{r+\ell}}.
 \]
In case $r = 0$, the same expression for $\adj_{m - r - \ell}(M)$
holds with\\
$\tM = (-1)^\ell \adjv
\begin{pmatrix}
 M_{21}
\end{pmatrix}
\adjv
\begin{pmatrix}
 M_{12}
\end{pmatrix}$.
\end{lemma}
\begin{proof} 
Suppose $r>0$. Recall that $X\in \R^{(m + \ell)\times (r+2\ell)}$, 
$Z\in \R^{(r+2\ell)\times (m + \ell)}$ and
$r + \ell \leq m$. Then $r + 2\ell \leq m + \ell$.
 For $\alpha,\beta\in \langle m+\ell\rangle_{m-r-\ell}$, 
we see that
\[
  M\left[\beta^c,\alpha^c\right]
= X[\beta^c,[r + 2\ell]]\, Y\, Z[[r+2\ell],\alpha^c].
\]
Note that $\#(\alpha^c) = \#(\beta^c) = r + 2\ell$.
Since $\det Y = (-1)^{\ell} \det D$, we obtain
\[
  \det M\left[\beta^c,\alpha^c\right]
= (-1)^\ell\det D\cdot\det X[\beta^c,[r+2\ell]]\cdot\det Z[[r+2\ell],\alpha^c].
\]
If $\alpha\cap [\ell] = \emptyset$, then 
\[
  \det Z[[r+2\ell],\alpha^c] 
= \det\left(
\left(
\begin{array}{ccc}
 & M_{12} & \\
 &  F &
\end{array}
 \right) 
[[r+\ell],\alpha^c\setminus[\ell]]\right).
\]
If $\alpha\cap [\ell] \neq \emptyset$, then
$\det Z[[r+2\ell], \alpha^c]=0$. 
Similar relations hold for $X[\beta^c,[r + 2\ell]]$.
Therefore, if $\alpha\cap [\ell] = \emptyset$ and $\beta\cap [\ell] = \emptyset$, then we obtain
\begin{align*}
&   \det M\left[\beta^c,\alpha^c\right]\\
& = (-1)^\ell\det D
 \cdot\det\left(
\left(
\begin{array}{@{}c@{\,}c@{}}
 \\
 M_{21} & C \\
 \phantom{1}
\end{array}
\right)
[\beta^c\setminus[\ell],[r+\ell]]\right)
\cdot
\det\left(
\left(
\begin{array}{ccc}
  & M_{12} & \\
 & F & 
\end{array}
\right)
[[r+\ell],\alpha^c\setminus[\ell]]\right). 
\end{align*}
Since $(-1)^{|\alpha| + |\beta|}\det M\left[\beta^c,\alpha^c\right]$ is $(\alpha,\beta)$-entry of $\adj_{m-r-\ell}(M)$, we have the equality.

In case $r = 0$, we have
$
 M = 
\left(\begin{smallmatrix}
 I & O \\
 O & M_{21}
\end{smallmatrix}\right)
\left(\begin{smallmatrix}
 M_{11} & I \\
 I & O
\end{smallmatrix}\right)
\left(\begin{smallmatrix}
 I & O \\
 O & M_{12}
\end{smallmatrix}\right)
$ and similar arguments for this relation give the result.
\end{proof}
We consider a symmetric matrix $M$ satisfying the following assumptions;
\begin{assumption}[$*$]
\hspace{1em}
 \begin{enumerate}
 \item $M \in \S^{\ell + m}$ is a matrix partitioned as $(\ref{eq:partition})$ for some $\ell, m \in \N$ with $M_{11}\in \S^{\ell}$ and $M_{12} = M_{21}^T$;
 \item $M_{21} = 
 \begin{pmatrix}
 b_1 & \cdots & b_{\ell}
 \end{pmatrix}\in \R^{m\times \ell}$ for some $b_k\in \R^{m}\ (k\in [\ell])$;
 \item 
 $M_{22} = CDC^T$ for some
 $r\in \N\cup \{0\}$, $C =
 \begin{pmatrix}
 c_1 & \cdots & c_r
 \end{pmatrix}
 \in \R^{m\times r}$, $D = \diag(u_1,\ldots,u_r)$,
 $u_k\in \R\ (k\in [r])$;

 \item $r + \ell \leq m$.
 \end{enumerate}
\end{assumption}
We see $CDC^T = \sum_{k=1}^r u_k c_kc_k^T$. Note that (i), (ii), (iii) give notation and (iv) is the required property.

\begin{prop}
\label{prop:sum_pminors}
Let $M\in \S^{\ell + m}$ satisfy the assumptions $(*)$.
In case $r>0$, we have
\[
 \sum_{(r+2\ell)\times (r+2\ell)} |M|
=
(-1)^\ell
u_1\cdots u_r
\sum_{(r+\ell)\times(r+\ell)}
\begin{vmatrix}
 b_1 & \cdots & b_\ell & c_1 & \cdots & c_r
\end{vmatrix}^2.
\]
In case $r = 0$, we have
$ \sum_{2\ell\times 2\ell} |M|
=
(-1)^\ell
\sum_{\ell\times\ell}
\begin{vmatrix}
 b_1 & \cdots & b_\ell 
\end{vmatrix}^2$.
\end{prop}
\begin{proof}
If $r>0$, Lemma $\ref{lemma:adj}$ implies that
$\adj_{m-r-\ell}(M) = 
\left(
\begin{smallmatrix}
 O & O \\
 O & \tM
\end{smallmatrix}
\right)$, where 
 \[
  \tM
= (-1)^\ell 
\det D\cdot
\adjv
\left(
\begin{array}{@{}c@{\,}c@{}}
 \\
 M_{21} & C\\
 \phantom{1}
\end{array}
\right)
\adjv
\left(
\begin{array}{ccc}
 &  M_{21}^T & \\
 &  C^T &
\end{array}
\right) \in \R^{\binom{m}{r+\ell}\times \binom{m}{r+\ell}}.
 \]
By the definitions, we see that the entries of $\adj_{m-r-\ell}(M)$ are minors of size $r + 2\ell$, 
and that the entries of the column vector 
$\adjv
\left(
\begin{array}{@{}c@{\,}c@{}}
 \\
 M_{21} & C\\
 \phantom{1}
\end{array}
\right)$
and the row vector 
$\adjv
\left(
\begin{array}{ccc}
 &  M_{21}^T & \\
 &  C^T &
\end{array}
\right)$
are minors of size $r + \ell$.
Since $\sum_{(r+2\ell)\times (r+2\ell)} |M| = \tr \adj_{m-r-\ell}(M) = \tr \tM$, we obtain the desired equality.
The case $r=0$ is shown similarly.
\end{proof}

\subsection{The Newton polytope}
\label{section:newton_polytope}

In this section, we will show that 
the coefficients of the expansion in Proposition $\ref{prop:expansion}$ are simplified  
if the terms correspond to some points on the boundary of the Newton polytope of the characteristic polynomial.
For a $k$-variable polynomial $f(x) = \sum_{\gamma}f_\gamma x^\gamma\in \R[x]$ where $x^\gamma = x_1^{\gamma_1}\cdots x_k^{\gamma_k}$,
let $\supp f = \{\gamma\in \Z^k:f_\gamma \neq 0\}$.
Then the convex hull of $\supp f$ is called the \textit{Newton polytope} of $f$. 
\begin{figure}[ht]
   \begin{tikzpicture}[scale=0.7]
    \draw[->](0,0) -- (6.5,0) node[right]{$\gamma_1$};
    \draw[->](0,0) -- (0,6.5) node[above]{$\gamma_2$};

    \draw[dashed](0,3)node[left]{$m$} -- (3,0)node[below]{$m$};
    \draw[very thick](0,6)node[left]{$n$} -- (6,0)node[below]{$n$};
    \draw[very thick](0,3) -- (1,2);
    \draw[very thick](1,2) -- (5,0);
    \draw[very thick](5,0) -- (6,0);
    \draw[very thick](0,3) -- (0,6);


    \draw (4.7,0) node[below]{$2m-r$};
    \draw (0,4.5) node[left]{$E_1$};
    \draw (0.4,2.4) node[below]{$E_2$};
    \draw (3.4,0.8) node[below]{$E_3$};

    \coordinate (O) at (1,2);
    \fill (O) circle (2pt);
    \draw (1.9,2.1) node[above] {$(r,m-r)$}; 
  \end{tikzpicture}
 \quad
  \begin{tikzpicture}[scale=0.7]
    \draw[->](0,0) -- (6.5,0) node[right]{$\gamma_1$};
    \draw[->](0,0) -- (0,6.5) node[above]{$\gamma_2$};

    \draw[dashed](0,4)node[left]{$m$} -- (4,0)node[below]{$m$};
    \draw[very thick](0,6)node[left]{$n$} -- (2.2,3.8);
    \draw[very thick](2.9,3.1) -- (5,1);
    \draw[very thick](0,4) -- (1,3);
    \draw[very thick](1,3) -- (5,1);
    \draw[dashed](5,1) -- (6,0)node[below]{$n$};
    \draw[very thick](0,4) -- (0,6);

    \coordinate (P) at (5,1);
    \fill (P) circle (2pt);    
    \draw (5,1) node[right]{$Q$};

    \draw (0,5) node[left]{$E_1$};
    \draw (0.4,3.4) node[below]{$E_2$};
    \draw (4,1.4) node[below]{$E_3$};

    \coordinate (O) at (1,3);
    \fill (O) circle (2pt);
    \draw (1.9,3.1) node[above] {$(r,m-r)$}; 
  \end{tikzpicture}
  \caption{$D_0$ and the Newton polytope, where $Q$ is $(2n-2m+r,2m-r-n)$. The left and the right figures correspond to the case $m-r < n-m$ and to the case $m-r> n-m$ respectively.}
  \label{fig:newton_polygon}
\end{figure}


For $m,n\in \N$ with $m<n$, and $a_k>0\ (k\in [n-m])$, let $A = \diag(a_1,\ldots, a_{n-m},0,\ldots,0)$, $B\in \S^n$. 
Then Proposition $\ref{prop:expansion}$ gives that
\[
  p_{A + tB}(x) =
\sum_{\gamma \in D_0}(-1)^{n - \gamma_2}\langle \adj_{|\gamma|}(A),C_{|\gamma|}^{\gamma_1}(B)\rangle t^{\gamma_1}x^{\gamma_2},
\]
where $|\gamma| = \gamma_1 + \gamma_2$. 
Consider the Newton polytope of $f(t,x) = p_{A + tB}(x)$.
We define
$D_0 = \{\gamma \in \Z^2:0 \leq \gamma_1 \leq n,\ m \leq \gamma_1 + \gamma_2 \leq n \}$.
Let $d_\gamma$ be the coefficient of $t^{\gamma_1}x^{\gamma_2}$ in $p_{A + tB}(x)$.
We will further calculate $d_{\gamma}$. 
%
Suppose that $B$ is partitioned as
\begin{equation}
  B = 
\begin{array}[b]{cl}
\phantom{(}
\begin{array}{ll}
 \mathmakebox[2em]{n-m} \hspace{2em}& \mathmakebox[2em]{m} \hspace{1em}
\end{array} 
\phantom{)}
&  \\
\left(
     \begin{array}{c|c}      
     B_{11} & B_{21}^T\\
      \hline
     B_{21} & B_{22}   
     \end{array} 
\right) 
& \hspace{-1em}
     \begin{array}{l}
      n-m \\
      m
     \end{array}
\end{array},\label{eq:partition0}
\end{equation}
and $B_{22} = CDC^T$ for some
 $r\in \N$, $C =
\begin{pmatrix}
 c_1 & \cdots & c_r
\end{pmatrix}
\in \R^{m\times r}\setminus\{O\}$, $D = \diag(u_1,\ldots,u_r)$,
$u_k\in \R\setminus\{0\}\ (k\in [r])$. Note that $\rank B_{22} = r$.
In the following Theorem, the expressions in (\ref{enum:0}), (\ref{enum:1}), (\ref{enum:2}) and (\ref{enum:3}) correspond to the cases that $\gamma$ is on the edge $E_1$, $E_2$, $E_3$ and below the edge $E_3$ in Figure $\ref{fig:newton_polygon}$, respectively. Recall that we defined $\sum_{0\times 0}|A|=1$.

\begin{thm}
\label{thm:generic}
Let $A = \diag(a_1,\ldots, a_{n-m},0,\ldots,0)$. 
\begin{enumerate}
\item
\label{enum:0}
If 
$\gamma = 
\begin{pmatrix}
 0\\
n
\end{pmatrix}
+ \eta
\begin{pmatrix}
 0\\
-1
\end{pmatrix},\ \eta = 0,\ldots,n-m$, then
$d_\gamma = 
 (-1)^{\eta}\sum_{\eta\times \eta}|A|.$
 \item 
\label{enum:1}
If $\gamma = 
\begin{pmatrix}
 0\\
m
\end{pmatrix}
+ \eta\begin{pmatrix}
   1\\
   -1
  \end{pmatrix},\ \eta = 1,\ldots,r
$, then
\[
d_\gamma
 =  (-1)^{n-m + \eta}\cdot a_1\cdots a_{n-m}\cdot \sum_{\eta\times \eta}
\begin{vmatrix}
 B_{22}
\end{vmatrix}.
\]
%

\item 
\label{enum:2}
If $r>0$ and 
$\gamma = 
\begin{pmatrix}
 r\\
m -r
\end{pmatrix}
+ \mu\begin{pmatrix}
   2\\
   -1
  \end{pmatrix},\ \mu = 1,\ldots, \min\{m-r,n-m\}
$, then
\begin{align*}
d_\gamma 
& = (-1)^{n-m+r+2\mu}\cdot u_1 \cdots u_r\cdot\\
& \sum_{\{k_1,\ldots,k_\mu \}\subset [n-m]}
\prod_{k\in [n-m]\setminus\{k_1,\ldots,k_\mu\}}a_{k}
  \sum_{(r+\mu)\times(r+\mu)}
  \begin{vmatrix}
   b_{k_1} & \cdots & b_{k_\mu} & c_1 & \cdots & c_r
  \end{vmatrix}^2. 
\end{align*}
If $r = 0$, then $u_1\cdots u_r$ is replaced by $1$ in the equality above.
\item 
\label{enum:3}
If $\gamma = 
\begin{pmatrix}
 r\\
m - r
\end{pmatrix}
+ \eta\begin{pmatrix}
   1\\
   -1
  \end{pmatrix} 
+ \mu\begin{pmatrix}
   2 \\
   -1
  \end{pmatrix}
,\ \mu = 0,\ldots,\min\{m-r,n-m\}-1,\ \eta = 1,\ldots, m-r-\mu$, then
$d_\gamma = 0$.
\end{enumerate}
\end{thm}
Before proceeding to the proof, we provide an illustrative example for 
$n=4,\ m = 2$ and $r = 1$.
\begin{example} 
\label{ex:newton_polytope}
 We consider $A,B$ in Example $\ref{ex:expansion}$.
In addition, we assume 
{\small 
$
 \begin{pmatrix}
  b_{33} & b_{34} \\
  b_{34} & b_{44}
 \end{pmatrix} = 
u_1
\begin{pmatrix}
  c_{11}\\
  c_{12}
 \end{pmatrix}
\begin{pmatrix}
	       c_{11} & c_{12}
	      \end{pmatrix}
$}.
Then Figure $\ref{fig:newton_polygon_ex}$ is the Newton polytope of $p_{A + tB}(x)$.
By applying Theorem $\ref{thm:generic}$, we can simplify the coefficients of terms corresponding to points on the lower and left boundary of the Newton polytope. 
As a result, we obtain the following expression:
\begin{multline*}
p_{A + tB}(x) = 
  x^4-\left(
a_1 + a_2 + O(t)
\right)x^3
+ (a_1 a_2 + O(t))x^2
 -\left(
u_1 a_1a_2 (c_{11}^2 + c_{12}^2)t + O(t^2)
\right)x \\
+
 u_1\left(
a_2 \begin{vmatrix}
 b_{13} & c_{11} \\
 b_{14} & c_{12}
 \end{vmatrix}^2 
+ a_1 \begin{vmatrix}
 b_{23} & c_{11} \\
 b_{24} & c_{12}
 \end{vmatrix}^2 
\right)t^3
+
\begin{vmatrix}
 B
\end{vmatrix}t^4.
\end{multline*}
\end{example}
\begin{figure}[ht]
    \begin{tikzpicture}[scale=0.7]
    \draw[->](0,0) -- (5,0) node[right]{$\gamma_1$};
    \draw[->](0,0) -- (0,5) node[above]{$\gamma_2$};

    \draw[dashed](0,2)node[left]{$2$} -- (2,0)node[below]{$2$};
    \draw[very thick](0,4)node[left]{$4$} -- (4,0)node[below]{$4$};
    \draw[very thick](0,2) -- (1,1);
    \draw[very thick](1,1) -- (3,0);
    \draw[very thick](3,0) -- (4,0);
    \draw[very thick](0,2) -- (0,4);


    \draw (3,0) node[below]{$3$};

    \coordinate (O) at (1,1);
    \fill (O) circle (2pt);
    \draw (0.8,1.4) node[right] {$(1,1)$}; 
  \end{tikzpicture}
  \caption{the Newton polytope in Example $\ref{ex:newton_polytope}$}
  \label{fig:newton_polygon_ex}
\end{figure}
\begin{proof}
Since $d_\gamma = (-1)^{n - \gamma_2}\langle \adj_{|\gamma|}(A),C_{|\gamma|}^{\gamma_1}(B)\rangle$, we calculate $\langle \adj_{|\gamma|}(A),C_{|\gamma|}^{\gamma_1}(B)\rangle$.

(\ref{enum:0}) Since $\gamma_1 = 0$, we have that $C_{|\gamma|}^{\gamma_1}(B)$ is the identity matrix and hence $\langle \adj_{|\gamma|}(A),C_{|\gamma|}^{\gamma_1}(B)\rangle = \tr \adj_{|\gamma|}(A) = \tr\adj_{n-\eta}(A) = \sum_{\eta\times \eta}|A|$.

(\ref{enum:1}) Since $|\gamma|=m$, we see that an entry of $\adj_{|\gamma|}(A)=\adj_m(A)$ is a cofactor of size $n-m$. 
Since $A = \diag(a_1,\ldots,a_{n-m},0,\ldots,0)$, the unique nonzero entry of the diagonal matrix $\adj_m(A)$ is 
$\det A\big[\{1,\ldots,n-m\}\big] = a_1\cdots a_{n-m}$.
If $\eta\neq 0$, then $ 
C_{|\gamma|}^{\gamma_1}(B)\big[\{n-m+1,\ldots,n\}\big] = 
\sum_{\gamma_1\times \gamma_1}
\begin{vmatrix}
 B\big[\{n-m+1,\ldots,n\}\big]
\end{vmatrix}=
\sum_{\eta\times \eta}
\begin{vmatrix}
 B_{22}
\end{vmatrix}$.
If $\eta = 0$, then 
$C_{|\gamma|}^{\gamma_1}(B)\big[\{n-m+1,\ldots,n\}\big] = 1$.

(\ref{enum:2}) Note that 
$n-m-\mu \geq n-m - \min\{m-r, n-m\} \geq \max\{n-2m+r,0\}\geq 0$ and that
an entry of $\adj_{|\gamma|}(A) = \adj_{m+\mu}(A)$ is 
a cofactor of size $n-m-\mu$. 
Since $A = \diag(a_1,\ldots,a_{n-m},0,\ldots,0)$, a nonzero entry of $\adj_{m+\mu}(A)$ is given as
\[
  \det A\big[[n-m]\setminus\{k_1,\ldots,k_{\mu}\}\big]
= \prod_{k\in [n-m]\setminus\{k_1,\ldots,k_\mu\}}a_{k},
\]
where $\{k_1,\ldots,k_\mu\}\subset [n-m]$.
Then the corresponding entry of $C_{|\gamma|}^{\gamma_1}(A) = C_{m+\mu}^{r+2\mu}(B)$ to $A\big[[n-m]\setminus\{k_1,\ldots,k_{\mu}\}\big]$ is
{\small %
\begin{align}
& 
\sum_{(r+2\mu)\times(r+2\mu)}
\begin{vmatrix}
 B\big[\{k_1,\ldots,k_\mu,n-m+1,\ldots,n\} \big]
\end{vmatrix}\label{eq:sum_minors}\\
& = \sum_{(r+2\mu)\times(r+2\mu)} \begin{vmatrix}
   &  & & & b_{k_1}^T & \\
  & B[\{k_1,\ldots,k_\mu\}]       & & & \vdots    &\\
  & & & & b_{k_\mu}^T &\\
              &        &            & &   \\             
  b_{k_1}     & \cdots & b_{k_\mu} & &  CDC^T 
              &        &            & &  
 \end{vmatrix}. \notag
\end{align}}%
Here, we see that $B\big[\{k_1,\ldots,k_\mu,n-m+1,\ldots,n\} \big]$ 
satisfies the assumptions $(*)$ 
since $r + \mu \leq m$ and the matrix has size $m + \mu$.
Thus Proposition $\ref{prop:sum_pminors}$ implies that $(\ref{eq:sum_minors})$ is equal to 
\[
 (-1)^\mu u_1 \cdots u_r
  \sum_{(r+\mu)\times(r+\mu)}
  \begin{vmatrix}
   b_{k_1} & \cdots & b_{k_\mu} & c_1 & \cdots & c_r
  \end{vmatrix}^2.
\]

(\ref{enum:3}) Since $|\gamma| = m + \mu$, a nonzero entry of $\adj_{|\gamma|}(A)$ is the one obtained in the proof of (\ref{enum:2}). The corresponding entry of 
$C_{|\gamma|}^{\gamma_1}(B)=C_{m+\mu}^{r+2\mu+\eta}(B)$ is
\begin{equation}
  \sum_{(r+2\mu+\eta)\times(r+2\mu+\eta)}
\begin{vmatrix}
 B\big[\{k_1,\ldots,k_\ell,n-m+1,\ldots,n\} \big]
\end{vmatrix},\label{eq:sum_minors0}
\end{equation}
which is the sum of the principal minors of $B\big[\{k_1,\ldots,k_\mu,n-m+1,\ldots,n\} \big]$ of size $r + 2\mu + \eta$, while the proof of (\ref{enum:2}) considers the principal minors of the same matrix of size $r + 2\mu$.
Since $\rank \begin{pmatrix}
 b_{k_1} & \cdots & b_{k_\mu} & CDC^T
\end{pmatrix} \leq r + \mu$, we see that
\[
 \rank B\big[\{k_1,\ldots,k_\mu,n-m+1,\ldots,n\} \big] 
\leq \mu + r + \mu < r + 2\mu + \eta.
\]
Thus all summands of $(\ref{eq:sum_minors0})$ are zero.
Therefore, we obtain $\langle\adj_{|\gamma|}(A), C_{|\gamma|}^{\gamma_1}(B)\rangle = 0$.

\end{proof}

%
%
Next, we consider the case that there exists $\gamma$ in the case (\ref{enum:2}) of Theorem $\ref{thm:generic}$ such that 
$\langle \adj_{|\gamma|}(A),C_{|\gamma|}^{\gamma_1}(B)\rangle = 0$.
In this case, the theorem below shows that
all the terms of $p_{A+tB}$ corresponding to the points
that are located right and lower to $\gamma$ 
in the plane of the Newton polytope are zero.

\begin{figure}[ht]
    \begin{tikzpicture}[scale=0.7]
    \draw[->](0,0) -- (6.5,0) node[right]{$\gamma_1$};
    \draw[->](0,0) -- (0,6.5) node[above]{$\gamma_2$};

    \draw[dashed](0,6)node[left]{$n$} -- (5,1);
    \draw[very thick](0,3) -- (1,2);
    \draw[very thick](1,2) -- (3,1);
    \draw[dashed](3,1) -- (5,1);
    \draw[very thick](0,3)node[left]{$m$} -- (0,6);


    \draw (0,4.5) node[left]{$E_1$};
    \draw (0.4,2.5) node[below]{$E_2$};
    \draw (1.5,1.6) node[below]{$E_3$};

    \coordinate (O) at (1,2);
    \fill (O) circle (2pt);
    \draw (1.9,2.1) node[above] {$(r,m-r)$}; 

    \coordinate (O) at (3,1);
    \fill (O) circle (2pt);
    \draw (3,0.9) node[below]{$(r+2\tmu, m-r-\tmu)$};
  \end{tikzpicture}
  \caption{the Newton polytope in a degenerate case}
  \label{fig:newton_polygon1}
\end{figure}

\begin{thm}
\label{thm:degenerate}
Let $A = \diag(a_1,\ldots,a_{n-m},0,\ldots,0)$ and $\tmu\in [\min\{m-r,n-m\}]$.
Then $d_{\tgamma} = 0$ for $\tgamma = 
\begin{pmatrix}
 r\\
m - r
\end{pmatrix}
+ \tmu\begin{pmatrix}
   2 \\
   -1
  \end{pmatrix}$ 
  if and only if
$\rank
\begin{pmatrix}
 B_{21} & B_{22}
\end{pmatrix}
 < r + \tmu$.
In this case, we have $d_\gamma = 0$ for
$\gamma = 
\begin{pmatrix}
 r\\
m - r
\end{pmatrix}
+ \mu\begin{pmatrix}
   2 \\
   -1
  \end{pmatrix}
+ 
\nu\begin{pmatrix}
 1\\
 0
 \end{pmatrix},\  \mu = \tmu,\ldots,\min\{m-r,n-m\},\ \nu = 0,1,\ldots,n - m - \mu
$.
\end{thm}
\begin{proof}
If $d_{\tgamma} = 
(-1)^{n - \tgamma_2}\langle \adj_{|\tgamma|}(A),C_{|\tgamma|}^{\tgamma_1}(B)\rangle
= 0$, then (\ref{enum:2}) of Theorem $\ref{thm:generic}$ gives that
\[
  \sum_{(r+\tmu)\times(r+\tmu)}
  \begin{vmatrix}
   b_{k_1} & \cdots & b_{k_{\tmu}} & c_1 & \cdots & c_r
  \end{vmatrix}^2 = 0 
\]
and hence $\rank 
\begin{pmatrix}
    b_{k_1} & \cdots & b_{k_{\tmu}} & c_1 & \cdots & c_r
\end{pmatrix} 
< r + \tmu$ 
for all $\{k_1,\ldots,k_{\tmu}\} \subset [n-m]$.
Suppose
$q: = \rank
\begin{pmatrix}
 B_{21} & B_{22}
\end{pmatrix} \geq r + \tmu$.
Since $c_1,\ldots, c_r$ are linearly independent and span the column space of $B_{22}$, 
there are 
$q - r$ column vectors $b'_{k_1}, \ldots, b'_{k_{q-r}}$ of $B_{21}$ such that $c_1,\ldots, c_r, b'_{k_1}, \ldots, b'_{k_{q-r}}$ form a basis for the column space of 
$
\begin{pmatrix}
 B_{21} & B_{22}
\end{pmatrix}
$. 
Since $q-r\geq \tmu$, we have 
$\rank 
\begin{pmatrix}
    b'_{k_1} & \cdots & b'_{k_{\tmu}} & c_1 & \cdots & c_r
\end{pmatrix} 
= r + \tmu$.
Thus
$\rank
\begin{pmatrix}
 B_{21} & B_{22}
\end{pmatrix} \geq r + \tmu$.
The converse is obvious.

Next, suppose 
$\rank
\begin{pmatrix}
 B_{21} & B_{22}
\end{pmatrix} < r + \tmu
$.
For $\mu = \tmu,\ldots, \min\{m-r,n-m\}$,
$\nu = 0,1,\ldots, n - m - \mu$, an entry of 
$\adj_{|\gamma|}(A) = \adj_{m+\mu+\nu}(A)$ is a cofactor of 
size $n - m - \mu - \nu$. 
Since $A = \diag(a_1,\dots,a_{n-m},0,\ldots,0)$, a nonzero entry of $\adj_{m+\mu+\nu}(A)$ is given as
\[
  \det A\big[[n-m]\setminus\{k_1,\ldots,k_{\mu+\nu}\}\big]
= \prod_{k\in [n-m]\setminus\{k_1,\ldots,k_{\mu+\nu}\}}a_{k}
\]
for some $\{k_1,\ldots,k_{\mu+\nu}\}\subset [\min\{m-r,n-m\}]$.
The corresponding entry of $C_{|\gamma|}^{\gamma_1}=C_{m+\mu+\nu}^{r+2\mu+\nu}(B)$ is 
\begin{equation}
  \sum_{(r+2\mu+\nu)\times(r+2\mu+\nu)}
\begin{vmatrix}
 B\big[\{k_1,\ldots,k_{\mu+\nu},n-m+1,\ldots,n\} \big]
\end{vmatrix}.\label{eq:sum_minors1}
\end{equation}
Since 
$\rank
\begin{pmatrix}
 B_{21} & B_{22}
\end{pmatrix} < r + \tmu$, we have
\[
 \rank
\begin{pmatrix}
 B\big[\{k_1,\ldots,k_{\mu+\nu},n-m+1,\ldots,n\} \big]
\end{pmatrix} < r + \tmu + \mu + \nu
\leq r + 2\mu + \nu.
\]
Thus all summands in $(\ref{eq:sum_minors1})$ are zero.
Therefore, we obtain $d_{\gamma} = (-1)^{n-\gamma_2}\langle\adj_{|\gamma|}(A), C_{|\gamma|}^{\gamma_1}(B)\rangle = 0$.
\end{proof}
\begin{example} 
\label{ex:degenerate}
Consider $p_{A+tB}(x)$ in Example $\ref{ex:newton_polytope}$.
If the coefficients 
$d_{(1,0)} =
u_1 a_2 
\left|
\begin{smallmatrix}
 b_{13} & c_{11} \\
 b_{14} & c_{12}
 \end{smallmatrix}
\right|^2 
+ 
u_1 a_1
\left|
\begin{smallmatrix}
 b_{23} & c_{11} \\
 b_{24} & c_{12}
 \end{smallmatrix}
\right|^2 = 0
$, then 
$d_{(2,0)} = 
\begin{vmatrix}
 B
\end{vmatrix} = 0.
$
\end{example}
\section{Leading terms of eigenvalues and the method of the Newton diagram}
\label{section:newton_diagram}
\subsection{The method of the Newton diagram}
For $A,B\in \S^n$ and a characteristic polynomial $p_{A+tB}(x)$, we consider the Newton diagram constructed in the following way.
First, draw the the Newton polytope of $p_{A+tB}$ as a polynomial in $(t,x)$.
Second, translate this Newton polytope so that 
the point $(0,n)$ coincides with the origin, and rotate it by $90^\circ$ degree.
Finally, extend the line segment containing the right most point until it intersects with the line $\gamma_2 = n$. 
The obtained diagram is called the \textit{Newton diagram} associated with $p_{A + tB}(x)$.
Figure $\ref{fig:newton_diagram}$ is the Newton diagram constructed from the Newton polytope in Figure $\ref{fig:newton_polygon1}$.

\begin{figure}[ht]
  \centering
  \begin{tikzpicture}[scale=0.7]
    \draw[->](0,0) -- (6.5,0) node[right]{$\gamma_2$};
    \draw[](0,0) -- (0,2.6); 
    \draw[->](0,3.4) -- (0,6)node[above]{$\gamma_1$};

    \draw[very thick](0,0) -- (3,0);
    \draw[very thick](3,0)node[below]{$n-m$} -- (4,1);
    \draw[very thick](4,1) -- (5,3);
    \draw[very thick, dashed](5,3) -- (6,5)node[left]{$(n,2m - r)$};
    \draw[dashed](6,0)node[below]{$n$} -- (6,6);
       
    \coordinate (O) at (0,0);
    \fill (O) circle (2pt); 
    \draw (0,0) node[below] {$O$}; 

    \coordinate (P1) at (3,0);
    \fill (P1) circle (2pt); 

    \coordinate (P0) at (5,3);
    \fill (P0) circle (2pt); 
    \draw (5,3) node[left]{$(n-m+r + \tmu, r + 2\tmu)$};

    \coordinate (P2) at (4,1);
    \fill (P2) circle (2pt);
    \draw (4,1.2) node[left] {$(n-m+r,r)$}; 

    \coordinate (P3) at (6,5);
    \fill (P3) circle (1.5pt); 

    \draw (1.5,0) node[below] {$E'_1$}; 
    \draw (4,0) node[above] {$E'_2$}; 
    \draw (4.5,2) node[right] {$E'_3$}; 

  \end{tikzpicture}
  \caption{The Newton diagram associated with the characteristic polynomial, which is obtained by rotating the left of Figure $\ref{fig:newton_polygon}$.}
  \label{fig:newton_diagram}
\end{figure}

It is well-known that the slope $d$ of the edge $\Gamma$ of the Newton diagram coincides with 
the leading degree of $k$ eigenvalues of $A+tB$, where $k$ is the length of the projection of $\Gamma$ onto $\gamma_2$ axis.
In addition, for $f(t,x) = p_{A + tB}(x)$, we write $f(t,x) = \sum_{\gamma} f_\gamma t^{\gamma_1}x^{\gamma_2}$
for some $f_\gamma\in \R$ and define $f_\Gamma(x) = \sum\{f_\gamma x^{\gamma_2}:\gamma \in \Gamma \cap \supp f\}$.
Then, the leading coefficients of the eigenvalues with leading degree $d$ are the nonzero solutions to the equation $f_\Gamma(x) = 0$; see, e.g. \cite{MBO}, \cite{MD}.

\begin{example}
\label{ex:diagram}
Consider 
 {\scriptsize
 \[
 A = 
 \begin{pmatrix}
 1 & 0 & 0 & 0 \\
 0 & 0 & 0 & 0 \\
  0 & 0 & 0 & 0 \\
 0 & 0 & 0 & 0 
 \end{pmatrix},\quad
 B = 
 \begin{pmatrix}
 0 & 0 & 1 & b \\
 0 & 2 & 1 & c \\
 1 & 1 & 1 & 0 \\
 b & c & 0 & 0
 \end{pmatrix}.
 \]}%
We demonstrate the method of the Newton diagram.
 Proposition $\ref{prop:expansion}$ gives
 {\small 
 \begin{multline*}
p_{A + tB}(x) = x^4 - \left(1 + 3t\right)x^3 + 
 \left( 3t  
 + 
\sum_{2\times 2}|B|
t^2
 \right)x^2 
 - 
 \left(
 (1 - c^2)t^2 + 
 \sum_{3\times 3}
 \begin{vmatrix}
 B
 \end{vmatrix}t^3
 \right)x\\
-c^2
 t^3
+ (-b^2 - 2bc + c^2)
t^4.
 \end{multline*}}

If $c\neq 0$, then the Newton diagram associated with $p_{A + tB}(x)$ is the left of Figure $\ref{fig:newton_diagram_ex}$. Let $\Gamma_1$ be the edge with slope $1$ in the figure, which tells that $A + tB$ has three eigenvalues with leading degree $1$, counted with their multiplicities. 
In addition, the leading coefficients of the eigenvalues $\lambda_1,\lambda_2,\lambda_3$ with leading degree $1$ are the nonzero solutions $\sigma_1,\sigma_2,\sigma_3$ to 
$f_{\Gamma_1}(x) = -x^3 + 3 x^2 - (1 - c^2)x  - c^2 = 0$; i.e., $\lambda_i = \sigma_i\,t + O(t^2),\ i =1,2,3$. The remaining eigenvalue is $\lambda_4 = 1 + O(t)$.
Note that the Newton diagram remains unchanged for $c=\pm 1$, as the leading coefficient of the $x^0$ term in $p_{A + tB}(x)$ is $-c^2t^3$.
Furthermore, if $c = 0$ and $b = 0$, then all the coefficients of the $x^0$ term are zero, and the Newton diagram remains unchanged, as shown in the center of Figure $\ref{fig:newton_diagram_ex}$.

If $c=0$ and $b \neq 0$, then the leading degree of the $x^0$ term is $4$ and the Newton diagram is the right of Figure $\ref{fig:newton_diagram_ex}$.
Let $\Gamma_2,\ \Gamma_3$ be the edges with slope $1$ and $2$ in the figure, respectively.
Then the leading coefficients of the eigenvalues $\lambda_1,\lambda_2$ with leading degree $1$ are the nonzero solutions to 
$f_{\Gamma_2}(x) = -x^3 + 3x^2 - x = 0$,
while the leading coefficient of the eigenvalue $\lambda_3$ with leading degree $2$ is the nonzero solution to $f_{\Gamma_3}(x) = -x - b^2 = 0$. Specifically, these eigenvalues are $\lambda_1 = \frac{3 - \sqrt{5}}{2}\,t + O(t^2),\ \lambda_2 = \frac{3 + \sqrt{5}}{2}\,t + O(t^2),\ \lambda_3 = -b^2\,t^2 + O(t^3)$.

\begin{figure}[ht]
  \centering
  \begin{tikzpicture}[scale=0.7]
   \draw[dotted] (0,0) grid (5,5);

    \draw[->](0,0) -- (5,0) node[right]{$\gamma_2$};
    \draw[->](0,0) -- (0,5)node[above]{$\gamma_1$};

    \draw[very thick](0,0) -- (1,0);
    \draw[very thick](1,0)node[below]{$1$} -- (4,3);
    \draw[dashed](4,0)node[below]{$4$} -- (4,5);
      
    \fill (0,0) circle (4pt); 
    \draw (0,0) node[below] {$O$}; 

    \fill (1,0) circle (4pt); 

    \fill  (2,1) circle (4pt); 

    \fill [fill=black!10!white] (3,2) circle (4pt); 
    \draw [thick] (3,2) circle (4pt);

    \fill (4,3) circle (4pt); 

    \draw (2.6,1.1) node[right] {$\Gamma_1$}; 

  \end{tikzpicture}
\quad
  \begin{tikzpicture}[scale=0.7]
   \draw[dotted] (0,0) grid (5,5);

    \draw[->](0,0) -- (5,0) node[right]{$\gamma_2$};
    \draw[->](0,0) -- (0,5)node[above]{$\gamma_1$};

    \draw[very thick](0,0) -- (1,0);
    \draw[very thick](1,0)node[below]{$1$} -- (3,2);
    \draw[very thick](3,2) -- (4,3);

    \draw[dashed](4,0)node[below]{$4$} -- (4,5);
      
    \coordinate (O) at (0,0);
    \fill (O) circle (4pt); 
    \draw (0,0) node[below] {$O$}; 

    \coordinate (P1) at (1,0);
    \fill (P1) circle (4pt); 

    \coordinate (P1) at (2,1);
    \fill (P1) circle (4pt); 

    \coordinate (P1) at (3,2);
    \fill (P1) circle (4pt);


    \draw (2.6,1.1) node[right] {$\Gamma_1$}; 

  \end{tikzpicture}
\quad
  \begin{tikzpicture}[scale=0.7]
   \draw[dotted] (0,0) grid (5,5);

    \draw[->](0,0) -- (5,0) node[right]{$\gamma_2$};
    \draw[->](0,0) -- (0,5)node[above]{$\gamma_1$};

    \draw[very thick](0,0) -- (1,0)node[below]{$1$};
    \draw[very thick](1,0) -- (3,2);
    \draw[very thick](3,2) -- (4,4);

    \draw[dashed](4,0)node[below]{$4$} -- (4,5);
      
    \coordinate (O) at (0,0);
    \fill (O) circle (2pt); 
    \draw (0,0) node[below] {$O$}; 

    \coordinate (P1) at (1,0);
    \fill (P1) circle (4pt); 

    \coordinate (P1) at (2,1);
    \fill (P1) circle (4pt); 

    \coordinate (P1) at (3,2);
    \fill (P1) circle (4pt);

    \coordinate (P1) at (4,4);
    \fill (P1) circle (4pt); 

    \draw (2.6,1.1) node[below] {$\Gamma_2$}; 
    \draw (3.1,2.2) node[right] {$\Gamma_3$}; 

  \end{tikzpicture}
  \caption{The Newton diagram associated with the characteristic polynomial in Example $\ref{ex:diagram}$. The case of $c\neq0$ corresponds to the left figure.
The point at $(2,3)$ might be located above $\Gamma_1$.
The case $c=0$ with $b = 0$ corresponds to the center figure, $c=0$ with $b \neq 0$ corresponds to the right figure.}
  \label{fig:newton_diagram_ex}
\end{figure}

\end{example}

\subsection{The Eigenvalues of a Perturbed Symmetric Matrix}

By applying the expansion formulas of the characteristic polynomial 
in Section $\ref{section:charpoly}$ and the method of the Newton diagram to a perturbed symmetic matrix, we obtain the following theorem.

\begin{thm}
\label{thm:degree}
Let $A\in \S^n_+$ and $B\in \S^n$.
Then every eigenvalue of $A + tB$ is an analytic function of $t$ whose leading term has degree less than or equal to $2$.
\end{thm}
\begin{proof}
It is well-known that every eigenvalue of $A + tB$ is an analytic function of $t$; see, e.g. \cite[Theorem $6.1$ in Chapter II]{K}.
The result obviously holds if $A$ has the full rank, $A = O$, or $B = O$.
Since the eigenvalues of $A + tB$ remain invariant under the transformation $Q\mapsto Q^T(A + tB)Q$ for an orthogonal matrix,
we may assume that $A = \diag(a_1,\ldots, a_{n-m},0,\ldots,0)$ for some $m\in \N$ with $m < n$ and $a_k>0\ (k\in [n-m])$.
In addition, we assume that $B$ is partitioned as in $(\ref{eq:partition0})$ and $B_{22} = CDC^T \in \S^{m}$ for some
 $r\in \N$, $C =
\begin{pmatrix}
 c_1 & \cdots & c_r
\end{pmatrix}
\in \R^{m\times r}$ with $\rank C = r$, $D = \diag(u_1,\ldots,u_r)$,
$u_k\in \R\setminus\{0\}\ (k\in [r])$. 
Note that $\rank B_{22} = r$.

In this proof, the Newton diagram associated with the characteristic polynomial $p_{A+tB}$ is simply called the Newton diagram.
We will show that every slope of the Newton diagram is less than or equal to $2$. 
By Proposition $\ref{prop:expansion}$, 
we can write $p_{A + tB}(x) = P_1(x) + P_2(x) + P_3(x)$, where
\begin{align*}
P_1(x) & = \sum_{i=m}^n \sum_{j=m-i}^{n-i}
 (-1)^{n-i}\langle \adj_{i+j}(A), C_{i+j}^j(B)\rangle t^jx^i,\\
P_2(x) & = \sum_{i=m-r}^{m-1} \sum_{j=m-i}^{n-i}
 (-1)^{n-i}\langle \adj_{i+j}(A), C_{i+j}^j(B)\rangle t^jx^i,\\
P_3(x) & =   \sum_{i=0}^{m-r-1} \sum_{j=m-i}^{n-i}
 (-1)^{n-i}\langle \adj_{i+j}(A), C_{i+j}^j(B)\rangle t^jx^i.
\end{align*}
Note that exponents of $(t,x)$ in $P_1, P_2$ and $P_3$ correspond to
points on the edges $E_1$, $E_2$ and $E_3$ in Figure $\ref{fig:newton_polygon}$.
Since $C_{i+j}^j(B)$ is the zero matrix for $j<0$, we see that
for $i=m,\ldots,n$,
\begin{align*}
&   \sum_{j=m-i}^{n-i}
 (-1)^{n-i}\langle \adj_{i+j}(A), C_{i+j}^j(B)\rangle t^jx^i\\
& = (-1)^{n-i}\langle \adj_{i}(A), C_{i}^0(B)\rangle x^i
+ \sum_{j=1}^{n-i}
 (-1)^{n-i}\langle \adj_{i+j}(A), C_{i+j}^j(B)\rangle t^jx^i.
\end{align*}
Here, $\langle \adj_{i}(A), C_{i}^0(B)\rangle = \tr\adj_{i}(A)>0$ and hence 
each coefficient of $x^m,x^{m+1},\ldots,x^n$, which appears in $P_1$, has a nonzero constant term.
Thus the Newton diagram has the vertices at $(0,0)$ and $(n-m,0)$.

If $r = m$, then $p_{A + tB}(x) = P_1(x) + P_2(x)$.
By (\ref{enum:1}) of Theorem $\ref{thm:generic}$, the leading term of $P_2(0)$ is 
\[
 (-1)^n\langle \adj_m(A), C_m(B)\rangle t^m= (-1)^n a_1\cdots a_{n-m}
\begin{vmatrix}
 B_{22}
\end{vmatrix}t^m,
\]
where $|B_{22}| \neq 0$.
Thus the Newton diagram has a vertex at $(n,m)$.
Since the slope of the line passing through $(n-m,0)$ and $(n,m)$ is $1$
and every eigenvalue of $A+tB$ is analytic, there is no vertex below the line.
Therefore, the vertices of the Newton diagram consist of $(0,0)$, $(n-m,0)$, $(n,m)$ and hence a slope of the Newton diagram is $0$ or $1$. 

Thus it is sufficient to consider that the case $r < m$. 
Now, $P_2$ has $x^{m-r}, x^{m-r+1},\ldots, x^{m-1}$ and their leading terms are
\[
 (-1)^{n-m+\eta}\langle \adj_m(A), C_m^{\eta}(B)\rangle t^{\eta} = (-1)^{n-m+\eta}a_1\cdots a_{n-m}\cdot \sum_{\eta \times \eta}
\begin{vmatrix}
 B_{22}
\end{vmatrix}
t^\eta,
\quad(\eta = 1,\ldots,r),
\]
by (\ref{enum:1}) of Theorem $\ref{thm:generic}$.
Thus $(n-m+r,r)$ is 
a possible vertex of the Newton diagram.
Next, we consider $P_3$ which has $x^{m-r-\mu}$ for $\mu = 1, \ldots, \min\{m-r,n-m\}$.
The leading term of $x^{m-r-\mu}$ corresponds to the point $\gamma = (m-r-\mu,r + 2\mu)$ in Figure $\ref{fig:newton_polygon}$,
 and (\ref{enum:2}) of Theorem $\ref{thm:generic}$ gives that
\begin{align}
& (-1)^{n-m+r+\mu} \langle\adj_{m+\mu}(A), C_{m+\mu}^{r+2\mu}(B)\rangle t^{r+2\mu} \label{eq:leading}\\
 = & (-1)^{n-m+r+2\mu}\cdot\sum_{\{k_1,\ldots,k_\mu \}\subset [n-m]}\left(\prod_{k\in \{k_1,\ldots,k_\mu\}^c}a_{k}\right)
  \cdot u_1 \cdots u_r \notag\\
&\phantom{(-1)^{n-m+r+2\mu}\cdot\sum_{\{k_1,\ldots,k_\mu \}\subset [n-m]}}
  \cdot\sum_{(r+\mu)\times(r+\mu)}
  \begin{vmatrix}
   b_{k_1} & \cdots & b_{k_\mu} & c_1 & \cdots & c_r
  \end{vmatrix}^2 t^{r+2\mu}.
\notag
\end{align}
Then Theorem $\ref{thm:degenerate}$ ensures that $(\ref{eq:leading})$ is $0$ if and only if 
$\rank
\begin{pmatrix}
 B_{12} & B_{22}
\end{pmatrix} < r + \mu.
$

If $\rank
\begin{pmatrix}
 B_{21} & B_{22}
\end{pmatrix} = r$, then Theorem $\ref{thm:degenerate}$ implies that all the coefficients of $x^0,x^1,\ldots,$ $x^{m-r-1}$ are zero. 
Thus the nonzero term of $p_{A + tB}$ with the lowest degree with respect to $x$ is $x^{m-r}$.
Since the line passing through $(n-m,0)$ and $(n-m+r,r)$ also passes through $(n,m)$,
 the vertices of the Newton diagram consist of $(0,0)$, $(n-m,0)$, $(n,m)$.
Therefore, a slope of the Newton diagram is $0$ or $1$.

If $\rank
\begin{pmatrix}
 B_{21} & B_{22}
\end{pmatrix} = r + \tmu$ for some $\tmu\in [\min\{m-r,n-m\}]$,
then Theorem $\ref{thm:degenerate}$ implies that
all the coefficients of 
$x^0, x^1,\ldots, x^{m-r-\tmu-1}$ are zero.
Here, we see that $x^{m-r-\tmu}$ corresponds to $(n-m+r+\tmu, r + 2\tmu)$ on the plane of the Newton diagram and this is the right most point corresponds to a nonzero term of $p_{A + tB}$. Thus the vertices of the Newton diagram
consist of $(0,0)$, $(n-m, 0)$, $(n - m + r, r)$, $(n, 2m-r)$.
Therefore, a slope of the Newton diagram is $0$, $1$, or $2$.
\end{proof}

By considering the case 
$\rank
\begin{pmatrix}
 B_{21} & B_{22}
\end{pmatrix} = r
$ in the proof, we obtain the following.
\begin{cor}
\label{cor:degree}
 For $m,n\in \N$ with $m<n$, and $a_k>0\ (k\in [n-m])$, let $A = \diag(a_1,\ldots, a_{n-m},0,\ldots,0)\in \S^n$ and $B\in \S^n$. 
Suppose that $B$ is partitioned as in $(\ref{eq:partition0})$.
If 
$\rank
\begin{pmatrix}
 B_{21} & B_{22}
\end{pmatrix} = 
\rank B_{22}$, then the leading coefficients of the eigenvalues $A + tB$ 
are less than or equal to $1$.
\end{cor}

\section{Convergence Analysis of the Alternating Projection Method}
\label{section:ap}

\subsection{Eigenvalue formula}
For $A\in \S^n_+\setminus \{O\}$ and $B\in \S^n\setminus\{O\}$, let $E = \{A + tB: t\in \R \}$ and
$\phi(t) = A+tB$. We show the following proposition which connects the eigenvalues of $A + tB$ and the convergence rate of alternating projections.
A formula is obtained for a general affine space in \cite{OSW}. We include a proof for the case that the affine space is a line for completeness of the paper. 
\begin{prop}
\label{prop:formula_eigen}
Let $\phi(t) = A+tB$, $T = \phi^{-1}\circ P_E\circ P_{\S^n_+} \circ \phi(t)$ and $\lambda_1(t), \ldots, \lambda_n(t)$ be the eigenvalues of $\phi(t)$. 
Suppose $E\cap \S^n_+ = \{A\}$.
Then we have
 \[
 T = t - \frac{1}{\|B\|^2}
 \sum_{i \in n(t)}\frac{d}{d t}\frac{1}{2}\lambda_i^2(t),
 \]
where $n(t) = \{i \in [n]:\lambda_i(t) < 0\}$ and $\|B\|$ is the Frobenius norm of $B$.
\end{prop}
\begin{proof} 
Let $V(t) = P_{\S^n_+}(\phi(t))$ and $v_i(t)$ be the orthonormal eigenvectors of 
$\phi(t)$ associated with $\lambda_i(t)$ for $i\in [n]$. Then 
$
 V(t) = \phi(t) - \sum_{i \in n(t)}\lambda_i(t)v_i(t)v_i(t)^T
$. 
 Since $P_E$ is the orthogonal projection onto $E$, we can easily show that
$
 P_E(V(t)) = U_* + \frac{\langle B, V(t) - U_*\rangle}{\|B\|^2}B
$. 
Thus 
\begin{align*}
  \phi(T) & = P_E(V(t)) = U_* + \frac{\langle B, \phi(t)\rangle}{\|B\|^2}B -
 \sum_{i\in n(t)}\lambda_i(t)\frac{\langle B, v_i(t)v_i(t)^T\rangle}{\|B\|^2}B
 - \frac{\langle B, U_*\rangle}{\|B\|^2}B\\
& = P_E(\phi(t)) - \sum_{i\in n(t)}\lambda_i(t)\frac{\langle B, v_i(t)v_i(t)^T\rangle}{\|B\|^2}B
 = \phi(t) - \sum_{i\in n(t)}\lambda_i(t)\frac{v_i(t)^T B v_i(t)}{\|B\|^2}B,
\end{align*}
and hence 
$
\displaystyle T = t - \sum_{i\in n(t)}\lambda_i(t)\frac{v_i(t)^T B v_i(t)}{\|B\|^2}
$. 
Since $A,B$ are symmetric, we see that $\lambda_i(t)$, $v_i(t)$ are analytic by \cite[Theorem 6.1 and Section 6.2 in Chapter II]{K} 
and $\frac{d}{dt}\lambda_i(t) = v_i(t)^TB v_i(t)$. This completes the proof.
\end{proof}

\subsection{Proof of Theorem $\ref{thm:convergence0}$}
\label{section:proof-conv}

Let $Q$ be an orthogonal matrix such that $Q^TAQ$ is equal to a diagonal matrix $\Lambda$.
Let $\tE = \{\Lambda + tQ^TBQ: t\in \R\}$.
Since the Frobenius norm remains invariant under the transformation 
$U\mapsto Q^TUQ$, we have
\begin{align*}
 \min_{U\in E} \|U - QVQ^T\| & = \min_{U\in E} \|Q^TUQ - V\| \\
&  = \min_{\tU\in \tE} \|\tU - V\| = \|P_{\tE}(V) - V\| 
= \|QP_{\tE}(V)Q^T - QVQ^T\|,
\end{align*}
and hence $P_{E}(QVQ^T) = QP_{\tE}(V)Q^T$.
Similarly, we have $P_{\S^n_+}(QUQ^T) = QP_{\S^n_+}(U)Q^T$.
Then, for $t_0\in \R$, there exists $t_1\in \R$ such that 
\[
 A + t_1 B = P_E\circ P_{\S^n_+}(A + t_0 B) = Q[P_{\tE} \circ P_{\S^n_+}(Q^T(A + t_0B) Q)]Q^T.
\]
Thus $\Lambda + t_1 Q^TBQ = P_{\tE} \circ P_{\S^n_+}(\Lambda + t_0 Q^TBQ)$.
Therefore, in the convergence analysis, we may assume that 
$A = \diag(a_1,\ldots,a_{n-m},0,\ldots,0) \in \R^{n\times n}$
for some $m\in \N$ with $m < n$ and $a_k>0\ (k\in [n-m])$. 
Then we obtain Theorem $\ref{thm:convergence0}$ as a direct application of Theorem $\ref{thm:convergence}$ below.
In the proof Theorem $\ref{thm:convergence}$, we use the following technical lemma, which is a slight modification of \cite[Lemma $3.1$]{OSW2023}.
\begin{lemma}
\label{lemma:cesaro}
Suppose that the sequence $\{x_k\}$ satisfies
\[
 x_{k+1} = x_k - \rho x_k^d + O(x_k^{d+1})\ (k = 0,1,\ldots)
\]
for some $\rho>0$, an odd positive integer $d$ that is greater than $1$. 
Then there exists $\delta>0$ such that if $0<|x_0|<\delta$, then $x_k$ has
the same sign as $x_0$ for $k=1,2,\ldots$
and $\{x_k\}$ converges to $0$ with $x_k = \Theta(k^{-\frac{1}{d-1}})$.
More precisely, we have
\[
 \lim_{k\to \infty}((d-1)\rho)^{\frac{1}{d-1}}k^{\frac{1}{d-1}}x_k = 1.
\]
%
\end{lemma}
\begin{proof}
For arbitrary $\epsilon>0$, there exists $\delta$ with $0< \delta < \rho$ such that
$0< x_k < \delta$ implies 
$0 < x_k + (- \rho - \epsilon)x_k^d < x_{k+1} < x_k + (- \rho + \epsilon) x_k^d < \delta$.
By induction, we have $x_k>0$ for $k=1,2,\ldots$.
Suppose $\alpha := \inf_k x_k >0$. 
Then $\alpha \leq x_{k+1} < x_k + (- \rho + \epsilon) \alpha^d$ and hence $\alpha < \alpha + (\rho - \epsilon)\alpha^d < x_k$.
This is a contradiction. Thus $\inf_k x_k = 0$. Since $\{x_k\}$ is a decreasing sequence, we obtain $x_k\to 0$. Then almost identical arguments in the proof of  \cite[Lemma $3.1$]{OSW2023} ensures that $\lim_{k\to \infty}((d-1) \rho)^{\frac{1}{d-1}}k^{\frac{1}{d-1}}x_k = 1$. 
In case that $x_1<0$, we put $x_k = -y_k$. Then we have $y_{k+1} = y_k - \rho y_k^d + O(y_k^{d+1})$ and $y_k>0$. Applying the argument above to $y_k$, we obtain the result.
\end{proof}

\begin{thm}
\label{thm:convergence}
Let $A = \diag(a_1,\ldots,a_{n-m},0,\ldots,0) \in \R^{n\times n}$
for some $m\in \N$ with $m < n$ and $a_k>0\ (k\in [n-m])$ and $B$ be partitioned as $(\ref{eq:partition0})$. 
Suppose $\S^n_+ \cap E = \{A\}$. Then
$\|U_{k} - A\| = O(k^{-\frac{1}{2}})$ for $U_{k+1} = P_E\circ P_{\S^n_+}(U_k)$ with $U_0\in E$. Moreover, if 
$\rank 
\begin{pmatrix}
 B_{21} & B_{22}
\end{pmatrix}
= \rank B_{22}
$,
then $\|U_{k} - A\|$ converges to $0$ in a linear rate.
\end{thm}
\begin{proof} 
Let $t_k = \phi^{-1}(U_k)$. Then $t_{k+1} = \phi^{-1}\circ P_E\circ P_{\S^n_+} \circ \phi(t_k)$.
By Theorem $\ref{thm:degree}$, the leading term of every eigenvalue of $A + tB$ has degree less than or equal to $2$. 
Then Proposition $\ref{prop:formula_eigen}$ gives that $t_{k+1} = t_k - \frac{1}{\|B\|^2}\sum_{i \in n(t_k)}\frac{d}{d t}\frac{1}{2}\lambda_i^2(t_k)= t_k - \rho t_k^d + O(t_k^{d+1})$ for some $\rho > 0$, $d\in \N$ with $d = 1$, or $3$.
If $d = 3$, then Lemma $\ref{lemma:cesaro}$ implies that 
$\lim_{k\to \infty}(2\rho)^{\frac{1}{2}}k^{\frac{1}{2}}x_k = 1$
and hence $\|U_k - A\|= \Theta(k^{-\frac{1}{2}})$.

Suppose $d = 1$ and estimate the value of $\rho$. 
Then only the eigenvalues of $A+tB$ with leading degree $1$ contribute to $\rho$.
By the Newton diagram method explained in the first paragraph of Section $\ref{section:newton_diagram}$ and (\ref{enum:1}) of Theorem $\ref{thm:generic}$, 
the leading coefficients of such eigenvalues are the nonzero solutions
to the equation 
\[
(-1)^{n-m}a_1\cdots a_{n-m}\left(x^m - \sum_{1 \times 1}|B_{22}|x^{m-1} + \sum_{2\times 2}|B_{22}|x^{m-2} + \cdots + (-1)^{m}\det B_{22}\right) = 0.
\] 
This is a constant multiple of the characteristic polynomial of $B_{22}$, and hence the leading coefficients are nonzero eigenvalues $\sigma_1,\ldots,\sigma_s$ of $B_{22}$. Thus the eigenvalues of $A+tB$ with leading degree $1$ can be written by $\lambda_i(t) = \sigma_i t + O(t^2)$ for $i = 1,\ldots s$.
We consider the case $t_k>0$. The other case is shown similarly.
Since the eigenvalues $\lambda_i(t)$ of $A + tB$ 
are analytic, there exists $I\subset [m]$ such that $n(t) = I$ for sufficiently small $t>0$.
Then we have 
\[
 \rho = \frac{1}{\|B\|^2}\sum_{i\in I}\left.\frac{d}{dt}\right|_{t=0}\frac{1}{2}\lambda_i^2(t) \leq \frac{\sum_{i=1}^s\sigma_i^2}{\|B\|^2} \leq \frac{\|B_{22}\|^2}{\|B\|^2}.
\]
If $\|B\|> \|B_{22}\|$, then $\rho < 1$.
If $\|B\| = \|B_{22}\|$, then $B_{22}$ is indefinite since $\S^n_+\cap E = \{A\}$. Thus $\rho < \frac{\|B_{22}\|^2}{\|B\|^2} = 1$.
Therefore $t_{k+1} = (1 - \rho) t_k + O(t_k^2)$ and hence
$\{t_k\}$ converges linearly.
In the case that 
$\rank\begin{pmatrix}
 B_{21} & B_{22}
\end{pmatrix} = 
\rank B_{22}$, 
Corollary $\ref{cor:degree}$ ensures $d = 1$.
\end{proof}

\begin{example}
\label{ex:convergence}
Consider the matrices $A,B$ in Example $\ref{ex:diagram}$
and $\{U_k\}$ constructed by the alternating projections 
$U_{k+1} = P_{\S^4_+}\circ P_E(U_k)$,
where $E = \{A + tB: t\in \R\}$.
Now
 $\rank
 \begin{pmatrix}
 B_{21} & B_{22}
 \end{pmatrix} = 
 \rank B_{22}$ if and only if 
 $b = 0$.
 In this case, Theorem $\ref{thm:convergence}$ implies that the convergence rate of $U_k$ is linear.
 In the other case, the convergence rate is $O(k^{-\frac{1}{2}})$.
 We can verify these convergence rates directly by examining the Newton diagram
 presented in Example $\ref{ex:diagram}$. We will further investigate this example with numerical experiments in Example $\ref{ex:rate}$.
\end{example}
\begin{example}
\label{ex:linear}
 The convergence rate given in Theorem $\ref{thm:convergence}$ is only an upper bound. Let $A$ be the same matrix in Example $\ref{ex:diagram}$ and 
 {\scriptsize
 \[
 B = 
 \begin{pmatrix}
 0 & 0 & 0 & 1 \\
 0 & 1 & 0 & 0 \\
 0 & 0 & -2 & 0 \\
 1 & 0 & 0 & 0
 \end{pmatrix}.
 \]}%
Then 
$\rank
\begin{pmatrix}
 B_{21} & B_{22}
\end{pmatrix} > \rank B_{22}$. 
However, the eigenvalues of $A + tB$ are $t, -2t, \frac{1 \pm \sqrt{1 + 4t^2}}{2}$.
Thus Proposition $\ref{prop:formula_eigen}$ implies that the convergence rate is linear.
\end{example}

\subsection{Singularity degree}
We consider sufficient conditions for 
the upper bound given in Theorem $\ref{thm:convergence}$ to be tight.
Here, we use the notion of the singularity degree of $E \cap \S^n_+$; see, e.g. \cite{BW,DW,SW}.
We give a brief explanation for the singularity degree.
Consider a general affine subspace 
\[
E' = \{X\in \S^n: \langle C_k, X \rangle = b_k\ (k\in [m])\},
\]
where $C_k\in \S^n\ (k\in [m]),\ b=(b_1,\ldots,b_m)^T\in \R^m$.
The face $F_{\min}$ of the convex set $\S^n_+$ is called the \textit{minimal face} of $E'\cap \S^n_+$ if $F_{\min}$ is the intersection of all the faces of $\S^n_+$ containing $E'\cap \S^n_+$.
Suppose that $E'$ intersects with $\S^n_+$ nontransversely; i.e.,
$E'\cap S^n_+\neq \emptyset$ and $E'\cap \intr\S^n_{+}= \emptyset$.
We can find $F_{\min}$ by the following procedure called the \textit{facial reduction}:
\begin{enumerate}
 \item Set $F_0 = \S^n_+,\ i = 1$;
 \item Find $y\in \R^m,\ U^i \in \S^n_+,\ V \in F_{i-1}^\perp$ such that
\[
    b^Ty = 0,\quad
 \sum_{k=1}^m y_k C_k = U^i + V
\notin F_{i-1}^\perp;
\]
\item Set $F_{i+1} = F_i \cap \{U^i\}^\perp$;
\item If $F_{i+1}=F_{\min}$, then output $F_{i+1}$ as $F_{\min}$.
Otherwise, set $i := i+1$ and repeat (ii), (iii), (iv). 

\end{enumerate}
It is well-known that only finitely many iterations are necessary to obtain $F_{\min}$.
Thus the iterative process can be expressed as
    \begin{align*}
     \S_+^n &= F_0 \overset{U^1}{\longrightarrow} F_1
   \overset{U^2}{\longrightarrow} F_2
   \overset{U^3}{\longrightarrow}
   \cdots
   \overset{U^s}{\longrightarrow} F_s=F_{\min}.
 \end{align*}
The minimum length $s$ of such sequences is called the \textit{singularity degree} of $E'\cap \S^n_+$.

First, we explain the relation between the singularity degree and $B_{22}$ by the following lemma.
\begin{lemma}
\label{lemma:sd2}
Suppose that 
$A = \diag(a_1,\ldots,a_{n-m},0,\ldots,0),\ a_k >0\ (k\in [n-m])$,
$B$ is partitioned as $(\ref{eq:partition0})$
 and $E\cap \S^n_+ = \{A\}$, where
$E = \{A + t B:t\in \R\}$.
\begin{enumerate}
 \item\label{lemma:enum1} If the singularity degree of $E\cap \S^n_+$ is greater than $1$, then $B_{22}$ is a nonzero positive or negative semidefinite matrix with $\det B_{22} = 0$.
\item\label{lemma:enum2} If $B_{22}$ is a nonzero positive or negative semidefinite matrix and 
$\rank
\begin{pmatrix}
 B_{21} & B_{22}
\end{pmatrix} = \rank B_{22}$, 
then $E\cap \S^n_+$ is not a singleton.
\end{enumerate} 
\end{lemma}
\begin{proof}
(\ref{lemma:enum1}) Obviously $B_{22}$ is nonzero. Suppose that $B_{22}$ is indefinite.
Then there exists an orthogonal matrix $Q$ such that $Q^T B_{22} Q = \Lambda$, 
where $\Lambda = \diag(\lambda_1,\ldots, \lambda_r,0,\ldots,0)$ and two of $\lambda_k$ have distinct signs. 
Then there exist positive numbers $\sigma_1,\ldots,\sigma_m$ such that $\sum_{k=1}^r  \sigma_k \lambda_k = 0$.
Define 
{\small $U = 
\left(
\begin{array}{c|c} 
 O & O \\
\hline
 O & Q\Sigma Q^T
\end{array}\right)$}, 
where the partition is the same as that of $B$ and $\Sigma = \diag(\sigma_1,\ldots,\sigma_m)$. 
Then we have $U\in \S^n_+$ and $\langle U, B\rangle = 0$. Let $N = \dim \S^n_+$. Since $E$ is a line, there exist linearly independent $C_1,\ldots, C_{N-1}$ such that $E = \{X \in \S^n:\langle C_1,X\rangle = 1, \langle C_k,X\rangle = 0, k = 2, \ldots, N-1\}$.  
Then $U \in \{B\}^\perp = \Span\{C_1,C_2, \ldots, C_{N-1}\}$ and hence
$U = \sum_{k=1}^{N-1}y_kC_k$ for some $y_k\in \R$.
In addition, we have
$y_1 = \sum_{k=1}^{N-1}y_k\langle C_k, A\rangle = 
\langle U, A\rangle = 0$ since $A\in E$. 
This means that the singularity degree is $1$. Therefore, we have shown that if the singularity degree is $2$ then $B_{22}$ is a nonzero positive or negative semidefinite matrix. In addition, $\det B_{22} = 0$ since $\intr(E\cap \S^n_+) = \emptyset$.

(\ref{lemma:enum2}) 
We assume that $B_{22}$ is a nonzero positive semidefinite matrix
since the other case is shown similarly.
Let $A$ be partitioned as $B$. 
For $t$ sufficiently close to $0$, we see that $A_{11}+ tB_{11}$ is positive definite. By considering the Schur complement of $A_{11} + tB_{11}$ in $A + tB$, we have that $A + tB$ is positive semidefinite if and only if $S(t):=tB_{22} - t^2 B_{21}(A_{11} + t B_{11})^{-1} B_{21}^T$ is positive semidefinite. 
Since 
$\rank
\begin{pmatrix}
 B_{21} & B_{22}
\end{pmatrix} = \rank B_{22}$, the column space of $B_{21}$ is contained in that of $B_{22}$ and hence contained in $\Span\{v_1, \ldots, v_r\}$, where $v_k\ (k\in [r])$ are eigenvectors of $B_{22}$ that are associated with positive eigenvalues. 
Let $v \in \Span\{v_1,\ldots,v_r\}$ and $w \in \ker B_{22}$. 
Since $\Span\{v_1,\ldots,v_r\}\perp \ker B_{22}$, we have
\[
(v + w)^TS(t)(v + w)  
 = t\left(v^T B_{22}v\right) - t^2 \left(v^T B_{21}(A_{11} + tB_{11})^{-1}B_{21}^T v\right).
\]
Since $v^T B_{22}v\geq \lambda_{\min}\|v\|^2$ where $\lambda_{\min}$ is a minimum positive eigenvalue of $B_{22}$, we obtain that $S(t)$ is positive semidefinite for sufficiently small $t>0$.
Therefore, $A + tB \in \S^n_+$ and hence $E\cap S^n_+$ is not a singleton.
%
\end{proof}
\begin{prop}
\label{prop:sd2}
Let $U_{k+1} = P_E\circ P_{\S^n_+}(U_k)$.
If $\S^n_+ \cap E = \{A\}$ and the singularity degree is greater than $1$, then there exists the initial point $U_0$ such that
$\|U_k - A\| = \Theta(k^{-\frac{1}{2}})$.
\end{prop}
\begin{proof} 
As explained in the first paragraph of Section $\ref{section:proof-conv}$, 
we may assume that
$A = \diag(a_1,\ldots,a_{n-m},0,\ldots,0),\ a_k >0\ (k\in [n-m])$ and
$B$ is partitioned as $(\ref{eq:partition0})$.
Furthermore, by Lemma $\ref{lemma:sd2}$, we may assume that $B_{22}$ is a nonzero positive semidefinite matrix, and we have $\rank
\begin{pmatrix}
 B_{21} & B_{22}
\end{pmatrix} > \rank B_{22}$. 
Let $r = \rank B_{22}$ and 
$d_{(i,j)}$ be the coefficient of $t^ix^j$ in
the characteristic polynomial $p_{A+tB}(x)$ of $A+t B$.
By (\ref{enum:1}) of Theorem $\ref{thm:generic}$ and the positive semidefiniteness of $B_{22}$, 
we obtain that
$\langle \adj_{|\gamma|}(A),C_{|\gamma|}^{\gamma_1}(B)\rangle>0$
for 
$\gamma \in \Gamma :=\{
\left(\begin{smallmatrix}
 0\\
m 
\end{smallmatrix}\right)
+ \eta
\left(\begin{smallmatrix}
   1\\
   -1
  \end{smallmatrix}\right),\ \eta = 0,\ldots, r\} 
$. 
Thus $d_{(0,m)}, d_{(1,m-1)}, \ldots, d_{(r,m-r)}$
are nonzero and have alternating signs. 
Let $f(x) = \sum_{\gamma\in \Gamma}d_\gamma x^{\gamma_2}$.
Then the equation $f(x)=0$ has no negative solution.
In fact, $f(-y)$ has no sign change in the coefficients, and hence $f(-y) = 0$ has no positive solution. 
For $\gamma\in \Gamma$, we see that 
the coefficient of $x^{\gamma_2}$ in $p_{A + tB}(x)$ is a polynomial in $t$
and its leading coefficient is $d_\gamma$.
Now, Theorem $\ref{thm:degenerate}$ ensures that 
$d_{(r+2,m-r-1)}$ is nonzero,
since 
$\rank
\begin{pmatrix}
 B_{21} & B_{22}
\end{pmatrix} > r$.
Together with (\ref{enum:3}) of Theorem $\ref{thm:generic}$,
we see that 
$\Gamma$ coincides with the set of all the integer points on the edge of slope $1$ of the Newton diagram associated with $p_{A+tB}$.
Thus, by the method of the Newton diagram explained in the first paragraph of Section $\ref{section:newton_diagram}$,
the leading coefficients of the eigenvalues of $A + tB$ whose
leading terms have degree $1$ are the nonzero solutions to the equation $f(x) = 0$.
Thus the leading term of an eigenvalue of $A + tB$ with a negative leading coefficient has degree greater than $1$, and hence Theorem $\ref{thm:degree}$ gives that the degree is $2$.
On the other hand, since $\S^n_+\cap E = \{A\}$, we see that $A + tB$ has at least one negative eigenvalue for $t\neq 0$.
For $U_0 = A + t_0 B$ for sufficiently small $t_0>0$, 
let $U_{k+1} = P_E\circ P_{\S^n_+}(U_k)$ and $t_k = \phi^{-1}(U_k)$.
Then Proposition $\ref{prop:formula_eigen}$ implies that 
$t_{k+1} = t_k - \frac{1}{\|B\|^2}\sum_{i \in n(t_k)}\frac{d}{d t}\frac{1}{2}\lambda_i^2(t_k)= t_k - \rho t_k^3 + O(t_k^4)$ for some $\rho > 0$.
Therefore, Lemma $\ref{lemma:cesaro}$ implies that
$\|U_k - A\| = \Theta(k^{-\frac{1}{2}})$.
%
%
\end{proof}

\begin{example} 
\label{ex:rate}
Let $A,B$ be matrices in Example $\ref{ex:diagram}$
and $U_{k+1} = P_E\circ P_{\S^4_+}(U_k)$, where $U_k = A + t_k B$ and $E = \{A + tB:t\in \R\}$.
First, consider the case $c= 0$ and $b \neq 0$. In this case, the singularity degree of 
the intersection of $\S^4_+$ and $E$ is $2$.
For simplicity, we investigate the case of $c=0$ and $b = 1$.
As calculated in Example $\ref{ex:diagram}$, the eigenvalues of $A+tB$ are 
$\lambda_1 = \frac{3 - \sqrt{5}}{2}\,t + O(t^2),\ \lambda_2 = \frac{3 + \sqrt{5}}{2}\,t + O(t^2),\ \lambda_3 = -t^2 + O(t^3),\ \lambda_4 = 1 + O(t)$.
If we choose $t_0 < 0$ sufficiently close to $0$, 
then we find that 
\[
 t_{k+1} = t_k - \frac{1}{11}\frac{d}{dt}\frac{1}{2}(\lambda_1^2(t_k) + \lambda_2^2(t_k)) = t_k - \frac{7}{11}t_k + O(t_k^2) = \frac{4}{11} t_k + O(t_k^2).
\]
Thus $\|U_k  - A\| \approx \|U_0\|\cdot\left(\frac{4}{11}\right)^k$; the convergence rate is linear. For $t_0 = -0.1$, our estimate gives 
$\|U_k  - A\| \approx 0.332\cdot(0.364)^k$, which is consistent with Figure $\ref{fig:Case1m}$.
On the other hand, 
if we choose $t_0 > 0$ sufficiently close to $0$, then we have
\[
 t_{k+1} = t_k - \frac{1}{11}\frac{d}{dt}\frac{1}{2}\lambda_3^2(t_k) = t_k - \frac{2}{11} t_k^3 + O(t_k^4).
\]
Then Lemma $\ref{lemma:cesaro}$ implies $\|U_k - A\| = \|B\|t_k \approx \frac{11}{2}k^{-1/2} = 5.5k^{-1/2}$. The left of Figure $\ref{fig:Case1p}$ illustrates that the convergence rate is $\Theta(k^{-1/2})$. The right of Figure $\ref{fig:Case1p}$ tells that the plot of $1/\|U_k - A\|^2$ approximately coincides with the line $74.51 + 0.034k$, which means $\|U_k - A\| \approx 5.42 k^{-1/2}$.
Combining the case $t_0 < 0$ and the case $t_0 > 0$, we conclude that 
the convergence rate is $O(k^{-1/2})$. Thus there is a gap between
the actual rate and the upper bound $O(k^{-1/6})$ derived from the singularity degree.

    
\caption{For the case $c=0$ and $b = 1$, the log-scale plot of $\|U_k-A\|$ with $t_0=-0.1$ and its fitting line.}
\label{fig:Case1m}
\end{figure}

\begin{figure}[ht]
 \centering
%
    
\caption{For the case $c=0$ and $b=1$, the left figure displays plots of $\sqrt{k}\|U_k-U_*\|$ and $\sqrt[6]{k}\|U_k-U_*\|$ with $t_0 = 0.1$. The right figure displays a plot of $1/\|U_k-U_*\|^2$ along with its fitting line.}
\label{fig:Case1p}
\end{figure}

Next, consider the case $c \neq 0$. Then the singularity degree of the intersection of $\S^4_+$ and $E$ is $1$. Specifically, we examine the case $c = \frac{1}{2}$ and $b = 0$.
As calculated in Example $\ref{ex:diagram}$, 
the eigenvalues of $A+tB$ are $\lambda_1 = \frac{1}{2}t + O(t^2),\ 
\lambda_2 = \frac{\sqrt{33} + 5}{4}t + O(t^2),\ 
\lambda_3 = \frac{-\sqrt{33} + 5}{4}t + O(t^2),\ 
\lambda_4 = 1 + O(t)$. Since the leading degrees of the eigenvalues are $0$ or $1$, $\|U_k - A\|$ converges to $0$ linearly for any initial point, which is consistent with Figure $\ref{fig:Case2}$.
Here again, the upper bound $O(k^{-1/2})$ derived from the singularity degree is not tight.

\begin{figure}[ht]
\centering
\begin{tikzpicture}[scale=0.75]
\begin{axis}[grid=major, xlabel={$k$}, legend entries={$\|U_{k}-A\|$, $0.320\times (0.996)^k$}, legend style={at={(axis cs:30,0.2745)}}]
\addplot[red] table [x=k, y=norm2] {
k	 norm2	 approx
   0	 3.2016e-01	 3.1995e-01
   1	 3.1881e-01	 3.1863e-01
   2	 3.1748e-01	 3.1732e-01
   3	 3.1615e-01	 3.1600e-01
   4	 3.1483e-01	 3.1470e-01
   5	 3.1351e-01	 3.1340e-01
   6	 3.1220e-01	 3.1210e-01
   7	 3.1090e-01	 3.1082e-01
   8	 3.0960e-01	 3.0953e-01
   9	 3.0831e-01	 3.0825e-01
  10	 3.0702e-01	 3.0698e-01
  11	 3.0575e-01	 3.0571e-01
  12	 3.0447e-01	 3.0445e-01
  13	 3.0321e-01	 3.0319e-01
  14	 3.0195e-01	 3.0194e-01
  15	 3.0069e-01	 3.0069e-01
  16	 2.9944e-01	 2.9945e-01
  17	 2.9820e-01	 2.9821e-01
  18	 2.9696e-01	 2.9698e-01
  19	 2.9573e-01	 2.9575e-01
  20	 2.9451e-01	 2.9453e-01
  21	 2.9329e-01	 2.9331e-01
  22	 2.9207e-01	 2.9210e-01
  23	 2.9086e-01	 2.9090e-01
  24	 2.8966e-01	 2.8969e-01
  25	 2.8846e-01	 2.8850e-01
  26	 2.8727e-01	 2.8731e-01
  27	 2.8609e-01	 2.8612e-01
  28	 2.8491e-01	 2.8494e-01
  29	 2.8373e-01	 2.8376e-01
  30	 2.8256e-01	 2.8259e-01
  31	 2.8140e-01	 2.8142e-01
  32	 2.8024e-01	 2.8026e-01
  33	 2.7909e-01	 2.7910e-01
  34	 2.7794e-01	 2.7795e-01
  35	 2.7680e-01	 2.7680e-01
  36	 2.7566e-01	 2.7565e-01
  37	 2.7453e-01	 2.7452e-01
  38	 2.7340e-01	 2.7338e-01
  39	 2.7228e-01	 2.7225e-01
  40	 2.7117e-01	 2.7113e-01
  41	 2.7005e-01	 2.7001e-01
  42	 2.6895e-01	 2.6889e-01
  43	 2.6785e-01	 2.6778e-01
  44	 2.6675e-01	 2.6668e-01
  45	 2.6566e-01	 2.6557e-01
  46	 2.6458e-01	 2.6448e-01
  47	 2.6349e-01	 2.6338e-01
  48	 2.6242e-01	 2.6230e-01
  49	 2.6135e-01	 2.6121e-01
  50	 2.6028e-01	 2.6013e-01
};
\addplot[blue, very thick, dotted] table [x=k, y=approx] {
k	 norm2	 approx
   0	 3.2016e-01	 3.1995e-01
   1	 3.1881e-01	 3.1863e-01
   2	 3.1748e-01	 3.1732e-01
   3	 3.1615e-01	 3.1600e-01
   4	 3.1483e-01	 3.1470e-01
   5	 3.1351e-01	 3.1340e-01
   6	 3.1220e-01	 3.1210e-01
   7	 3.1090e-01	 3.1082e-01
   8	 3.0960e-01	 3.0953e-01
   9	 3.0831e-01	 3.0825e-01
  10	 3.0702e-01	 3.0698e-01
  11	 3.0575e-01	 3.0571e-01
  12	 3.0447e-01	 3.0445e-01
  13	 3.0321e-01	 3.0319e-01
  14	 3.0195e-01	 3.0194e-01
  15	 3.0069e-01	 3.0069e-01
  16	 2.9944e-01	 2.9945e-01
  17	 2.9820e-01	 2.9821e-01
  18	 2.9696e-01	 2.9698e-01
  19	 2.9573e-01	 2.9575e-01
  20	 2.9451e-01	 2.9453e-01
  21	 2.9329e-01	 2.9331e-01
  22	 2.9207e-01	 2.9210e-01
  23	 2.9086e-01	 2.9090e-01
  24	 2.8966e-01	 2.8969e-01
  25	 2.8846e-01	 2.8850e-01
  26	 2.8727e-01	 2.8731e-01
  27	 2.8609e-01	 2.8612e-01
  28	 2.8491e-01	 2.8494e-01
  29	 2.8373e-01	 2.8376e-01
  30	 2.8256e-01	 2.8259e-01
  31	 2.8140e-01	 2.8142e-01
  32	 2.8024e-01	 2.8026e-01
  33	 2.7909e-01	 2.7910e-01
  34	 2.7794e-01	 2.7795e-01
  35	 2.7680e-01	 2.7680e-01
  36	 2.7566e-01	 2.7565e-01
  37	 2.7453e-01	 2.7452e-01
  38	 2.7340e-01	 2.7338e-01
  39	 2.7228e-01	 2.7225e-01
  40	 2.7117e-01	 2.7113e-01
  41	 2.7005e-01	 2.7001e-01
  42	 2.6895e-01	 2.6889e-01
  43	 2.6785e-01	 2.6778e-01
  44	 2.6675e-01	 2.6668e-01
  45	 2.6566e-01	 2.6557e-01
  46	 2.6458e-01	 2.6448e-01
  47	 2.6349e-01	 2.6338e-01
  48	 2.6242e-01	 2.6230e-01
  49	 2.6135e-01	 2.6121e-01
  50	 2.6028e-01	 2.6013e-01
};
\end{axis}
\end{tikzpicture}
\quad
\begin{tikzpicture}[scale=0.75]
\begin{axis}[grid=major, xlabel={$k$},ymode=log, legend entries={$\|U_{k}-A\|$, $0.270\times (0.214)^k$}, legend style={at={(axis cs:30,10e-27)}}]
\addplot[red] table [x=k, y=norm2] {
k	 norm2	 approx
   0	 3.2016e-01	 2.7037e-01
   1	 5.9955e-02	 5.7905e-02
   2	 1.2494e-02	 1.2402e-02
   3	 2.6603e-03	 2.6561e-03
   4	 5.6906e-04	 5.6887e-04
   5	 1.2185e-04	 1.2184e-04
   6	 2.6095e-05	 2.6094e-05
   7	 5.5887e-06	 5.5887e-06
   8	 1.1970e-06	 1.1970e-06
   9	 2.5636e-07	 2.5636e-07
  10	 5.4905e-08	 5.4905e-08
  11	 1.1759e-08	 1.1759e-08
  12	 2.5185e-09	 2.5185e-09
  13	 5.3939e-10	 5.3939e-10
  14	 1.1552e-10	 1.1552e-10
  15	 2.4742e-11	 2.4742e-11
  16	 5.2991e-12	 5.2991e-12
  17	 1.1349e-12	 1.1349e-12
  18	 2.4307e-13	 2.4307e-13
  19	 5.2060e-14	 5.2060e-14
  20	 1.1150e-14	 1.1150e-14
  21	 2.3880e-15	 2.3880e-15
  22	 5.1144e-16	 5.1144e-16
  23	 1.0954e-16	 1.0954e-16
  24	 2.3460e-17	 2.3460e-17
  25	 5.0245e-18	 5.0245e-18
  26	 1.0761e-18	 1.0761e-18
  27	 2.3048e-19	 2.3048e-19
  28	 4.9362e-20	 4.9362e-20
  29	 1.0572e-20	 1.0572e-20
  30	 2.2642e-21	 2.2642e-21
  31	 4.8494e-22	 4.8494e-22
  32	 1.0386e-22	 1.0386e-22
  33	 2.2244e-23	 2.2244e-23
  34	 4.7642e-24	 4.7642e-24
  35	 1.0204e-24	 1.0204e-24
  36	 2.1853e-25	 2.1853e-25
  37	 4.6804e-26	 4.6804e-26
  38	 1.0024e-26	 1.0024e-26
  39	 2.1469e-27	 2.1469e-27
  40	 4.5981e-28	 4.5981e-28
  41	 9.8480e-29	 9.8480e-29
  42	 2.1092e-29	 2.1092e-29
  43	 4.5173e-30	 4.5173e-30
  44	 9.6748e-31	 9.6748e-31
  45	 2.0721e-31	 2.0721e-31
  46	 4.4379e-32	 4.4379e-32
  47	 9.5048e-33	 9.5048e-33
  48	 2.0357e-33	 2.0357e-33
  49	 4.3599e-34	 4.3599e-34
  50	 9.3377e-35	 9.3377e-35
};
\addplot[blue, very thick, dotted] table [x=k, y=approx] {
k	 norm2	 approx
   0	 3.2016e-01	 2.7037e-01
   1	 5.9955e-02	 5.7905e-02
   2	 1.2494e-02	 1.2402e-02
   3	 2.6603e-03	 2.6561e-03
   4	 5.6906e-04	 5.6887e-04
   5	 1.2185e-04	 1.2184e-04
   6	 2.6095e-05	 2.6094e-05
   7	 5.5887e-06	 5.5887e-06
   8	 1.1970e-06	 1.1970e-06
   9	 2.5636e-07	 2.5636e-07
  10	 5.4905e-08	 5.4905e-08
  11	 1.1759e-08	 1.1759e-08
  12	 2.5185e-09	 2.5185e-09
  13	 5.3939e-10	 5.3939e-10
  14	 1.1552e-10	 1.1552e-10
  15	 2.4742e-11	 2.4742e-11
  16	 5.2991e-12	 5.2991e-12
  17	 1.1349e-12	 1.1349e-12
  18	 2.4307e-13	 2.4307e-13
  19	 5.2060e-14	 5.2060e-14
  20	 1.1150e-14	 1.1150e-14
  21	 2.3880e-15	 2.3880e-15
  22	 5.1144e-16	 5.1144e-16
  23	 1.0954e-16	 1.0954e-16
  24	 2.3460e-17	 2.3460e-17
  25	 5.0245e-18	 5.0245e-18
  26	 1.0761e-18	 1.0761e-18
  27	 2.3048e-19	 2.3048e-19
  28	 4.9362e-20	 4.9362e-20
  29	 1.0572e-20	 1.0572e-20
  30	 2.2642e-21	 2.2642e-21
  31	 4.8494e-22	 4.8494e-22
  32	 1.0386e-22	 1.0386e-22
  33	 2.2244e-23	 2.2244e-23
  34	 4.7642e-24	 4.7642e-24
  35	 1.0204e-24	 1.0204e-24
  36	 2.1853e-25	 2.1853e-25
  37	 4.6804e-26	 4.6804e-26
  38	 1.0024e-26	 1.0024e-26
  39	 2.1469e-27	 2.1469e-27
  40	 4.5981e-28	 4.5981e-28
  41	 9.8480e-29	 9.8480e-29
  42	 2.1092e-29	 2.1092e-29
  43	 4.5173e-30	 4.5173e-30
  44	 9.6748e-31	 9.6748e-31
  45	 2.0721e-31	 2.0721e-31
  46	 4.4379e-32	 4.4379e-32
  47	 9.5048e-33	 9.5048e-33
  48	 2.0357e-33	 2.0357e-33
  49	 4.3599e-34	 4.3599e-34
  50	 9.3377e-35	 9.3377e-35
};
\end{axis}
\end{tikzpicture}
\caption{For the case $c = \frac{1}{2}$ and $b = 0$,
the left figure displays the log-scale plot of $\|U_k-A\|$ with $t_0 = 0.1$ and its fitting curve, while
the right figure displays the log-scale plot of $\|U_k-A\|$ with $t_0 = -0.1$ and its fitting curve.}
\label{fig:Case2}
\end{figure}

\end{example}

\section{Acknowledgment}
The first author was supported by JSPS KAKENHI Grant Number JP19K03631 and JP24K06841. 
The second author was supported by JSPS KAKENHI Grant Number JP17K18726 
and JSPS Grant-in-Aid for Transformative Research Areas (A) (22H05107). 
The third author was supported by JSPS KAKENHI Grant Number JP24K14843.\\

\noindent\textbf{Data availability.} 
Data sets generated during the current study are available from the corresponding author on reasonable request.

\section*{Declarations}

\noindent\textbf{Conflict of interest.}
On behalf of all authors, the corresponding author states that there is no conflict of interest.

\end{document}